\RequirePackage{fix-cm}
\documentclass[smallextended]{svjour3}       
\smartqed  
\usepackage{amssymb}
\usepackage{bm}
\usepackage{graphicx}
\usepackage{graphics}
\usepackage{psfrag}
\usepackage{arydshln}
\usepackage{amsmath}
\usepackage{amscd}
\usepackage{amsfonts}
\usepackage{booktabs}
\usepackage{float}
\usepackage{verbatim}
\usepackage{latexsym}
\usepackage{color}
\usepackage{multirow}
\usepackage{lscape}
\usepackage{arydshln}
\usepackage{epstopdf}
\usepackage{extarrows}

\graphicspath{{Fig/}}

\usepackage[left]{lineno}

\newcommand{\om}{\Omega}
\newcommand{\ve}{\varepsilon}
\def\C{\mathbb{C}}
\newcommand{\p}{\partial}

\newcommand{\ds}{\displaystyle}
\newcommand{\f}{\frac}

\begin{document}

\title{Numerical analysis of high-order methods for variable-exponent fractional diffusion-wave equation}

\titlerunning{Variable-exponent fractional diffusion-wave equation}        

\author{ Wenlin Qiu\textsuperscript{1}  \and Xiangcheng Zheng\textsuperscript{1}
}

\authorrunning{W. Qiu   \and X. Zheng } 

\institute{
            W. Qiu   \\
           \email{wlqiu@sdu.edu.cn} \\
           \at
           X. Zheng (Corresponding author) \\
           \email{xzheng@sdu.edu.cn} \\
           \at
           {1} School of Mathematics, Shandong University, Jinan, Shandong 250100, P. R. China
}

\date{Received: date / Accepted: date}

\maketitle

\begin{abstract}
This work considers the variable-exponent fractional diffusion-wave equation, which describes, e.g. the propagation of mechanical diffusive waves in viscoelastic media with varying material properties. Rigorous numerical analysis for this model is not available in the literature, partly because the variable-exponent Abel kernel in the leading term may not be positive definite or monotonic. We adopt the idea of model reformulation to obtain a more tractable form, which, however, still involves an ``indefinite-sign, nonpositive-definite, nonmonotonic'' convolution kernel that introduces difficulties in numerical analysis. We address this issue to design two high-order schemes and derive their stability and error estimate based on the proved solution regularity, with $\alpha(0)$-order and second-order accuracy in time, respectively.  Numerical experiments are presented to substantiate the theoretical findings and to show the transition behavior of the mechanical (diffusive) wave modeled by variable exponent.

\keywords{fractional, variable exponent, regularity, error estimate, second order}

\subclass{35L05 \and 65M12 \and 65M60}
\end{abstract}

\section{Introduction}
We consider the variable-exponent fractional diffusion-wave equation with $1<\alpha(t)<2$
\begin{align}\label{VtFDEs}
^c\p_t^{\alpha(t)} u(\bm x,t)- \kappa   A u (\bm x,t)= f(\bm x,t) ,~~(\bm x,t) \in \Omega\times(0,T],
\end{align}
equipped with initial and boundary conditions
\begin{equation}\label{ibc}
 u(\bm x,0)=u_0(\bm x),~\p_tu(\bm x,0)=\bar u_0(\bm x),~\bm x\in \Omega; \quad u(\bm x,t) = 0,~(\bm x,t) \in \p \Omega\times[0,T]. \end{equation}
Here $\kappa>0$, $A$ refers to the Laplacian, $\Omega \subset \mathbb{R}^d$ is a simply-connected bounded domain with the piecewise smooth boundary $\p \om$ with convex corners, $\bm x := (x_1,\cdots,x_d)$ with $1 \le d \le 3$ denotes the spatial variable, $f$ and $u_0$ refer to the source term and the initial value, respectively, and the fractional derivative is defined by the symbol $*$ of convolution \cite{Pat,Sun}
\begin{align}\label{var_RL}
^c\p_t^{\alpha(t)} u:=(k*\p_t^2 u)=\int_0^tk(t-s)\p_s^2u(\bm x,s)ds,~~ k(t):=\frac{t^{1-\alpha(t)}}{\Gamma(2-\alpha(t))}.
\end{align}
In general, the $- A u$ in (\ref{VtFDEs}) could be extended to $-\nabla(K(\bm x)\nabla u)+c(\bm x)u$ for the symmetric, positive definite diffusivity tensor $K(\bm x)$ and the non-negative reaction coefficient   $c(\bm x)$ without affecting the analysis, and we use $- A u$ for simplicity.

The constant-exponent fractional diffusion-wave equation has been used in various fields, for instance, it governs the propagation of mechanical diffusive waves in viscoelastic media which exhibit a power-law creep \cite{Mai,Mai95}. Thus, there exist significant progresses on its mathematical and numerical studies, see e.g. \cite{Banjai,DElia,Du,Fan,Kop,Li1,Li2,LiaTan,McLean,Mus1,Suz2,Xu,ZhaZho}.
For the variable-exponent  fractional diffusion-wave equation, a practical application is developed in \cite{ZhaSun} to apply the variable-exponent differential operators to switch from a wave- to a diffusion-dominated equation at the
boundaries, in order ``to control errors in the solutions of classical integer-order partial differential equations arising from truncated domains and erroneous boundary conditions or from the loss of monotonicity of the numerical solution either because of under-resolution or because of the presence of discontinuities'', as commented in the review paper \cite{Pat}. In \cite[Page 344]{Cai}, it is further demonstrated for this application that it is necessary for the variable exponent to be time-dependent to achieve such switching, which justifies the time dependence of the variable exponent in (\ref{VtFDEs}). Mathematically, the time-dependent variable exponent makes the fractional operators nonpositive-definite and nonmonotonic, which,  compared with the space-dependent variable exponent, introduces significant difficulties.  For these reasons, we consider the time-dependent case $\alpha=\alpha(t)$ in model (\ref{VtFDEs}).

There exist sophisticated numerical studies on variable-exponent time/space-fractional diffusion problems \cite{Moh,ZenZhaKar,Zhao,ZhuLiu}, while rigorous analysis for diffusion-wave type equations such as model (\ref{VtFDEs})--(\ref{ibc}) is not available in the literature.
 For variable-exponent fractional diffusion-wave models containing an additional temporal leading term $\p_t^2u$, the mathematical analysis and a first-order-in-time scheme are recently proposed in \cite{ZheWan}, and a second-order finite difference scheme is analyzed in \cite{DuSun}. For model (\ref{VtFDEs})--(\ref{ibc}) where $^c\p_t^{\alpha(t)} u$ serves as the leading term, the existing methods do not apply directly. The inherent reason is that the variable-exponent Abel kernel $k(t)$ could not be analytically treated by, e.g. the integral transform, and may not be positive definite or monotonic. Indeed, it is commented in \cite{Kia} that the analysis of variable-exponent subdiffusion remains an open problem, and so does the variable-exponent fractional diffusion-wave equation.

Recently, a convolution method is developed in \cite{Zhe}, which converts the variable-exponent subdiffusion to more feasible formulations such that the corresponding analysis becomes tractable. We employ this idea to convert model (\ref{VtFDEs})--(\ref{ibc}) to a transformed equation, based on which we address its well-posedness and solution regularity.  For numerical analysis, a main difficulty is that a convolution term in the transformed model involves an ``indefinite-sign, nonpositive-definite, nonmonotonic'' convolution kernel, which restricts the choice of quadrature rules in constructing high-order schemes and deteriorates the structures of numerical schemes. 

To {\color{black}overcome} the difficulty, we approximate this {\color{black}complex} term with the help of piecewise linear interpolation. Then we design suitable combinations of different discretization methods to approximate multiple terms in the transformed equation such that the $\alpha(0)$-order and second-order accuracy in time could be achieved and proved under a very weak constraint on $\alpha(t)$. Numerical experiments indicate that both methods exhibit the desired accuracy, even the imposed constraint on $\alpha(t)$ is not satisfied. Furthermore, the transition behavior of the mechanical (diffusive) wave is simulated to show the advantage of the variable-exponent model.

{\color{black}The rest of the work is organized as follows:} In Section 2 we introduce preliminary results and the mathematical analysis of the proposed model. In Sections 3--4, the $\alpha(0)$-order and second-order schemes are developed and analyzed, respectively, and their accuracy is substantiated in Section \ref{sec5}. We finally address concluding remarks in the last section.

\section{Mathematical analysis}
\subsection{Notations}
For the sake of clear presentation, we introduce the following spaces and notations.
Let $L^p(\om)$ with $1 \le p \le \infty$ be the Banach space of $p$th power Lebesgue integrable functions on $\om$. Denote the inner product of $L^2(\Omega)$ as $(\cdot,\cdot)$. For a positive integer $m$,
let  $ W^{m, p}(\Omega)$ be the Sobolev space of $L^p$ functions with $m$th weakly derivatives in $L^p(\om)$ (similarly defined with $\om$ replaced by an interval $\mathcal I$). Let  $H^m(\Omega) := W^{m,2}(\Omega)$ {\color{black} and $H^m_0(\Omega)$ be the closure of $C_0^{\infty}(\Omega)$ in $H^{m}(\Omega)$.} For a non-integer $s\geq 0$, $H^s(\Omega)$ is defined via interpolation \cite{AdaFou}. Let $\{\lambda_i,\phi_i\}_{i=1}^\infty$ be eigen-pairs of the problem $- A \phi_i = \lambda_i \phi_i$ with the zero boundary condition. We introduce the Sobolev space $\check{H}^s(\Omega)$ for $s\geq 0$ by
$\check{H}^{s}(\Omega) := \big \{ q \in L^2(\Omega): \| q \|_{\check{H}^s}^2 : = \sum_{i=1}^{\infty} \lambda_i^{s} (q,\phi_i)^2 < \infty \big \},$
which is a subspace of $H^s(\Omega)$ satisfying $\check{H}^0(\Omega) = L^2(\Omega)$ and $\check{H}^2(\Omega) = H^2(\Omega) \cap H^1_0(\Omega)$ \cite{Tho}.
For a Banach space $\mathcal{X}$, let $W^{m, p}(0,T; \mathcal{X})$ be the space of functions in $W^{m, p}(0,T)$ with respect to $\|\cdot\|_{\mathcal {X}}$.  All spaces are equipped with standard norms \cite{AdaFou}.

We use $Q$ to denote a generic positive constant that may assume different values at different occurrences. We use $\|\cdot\|$ to denote the $L^2$ norm of functions or operators in $L^2(\Omega)$, set $L^p(\mathcal X)$ for $L^p(0,T;\mathcal X)$ for brevity, and drop the notation $\om$ in the spaces and norms if no confusion occurs. For instance, $L^2(L^2)$ implies $L^2(0,T;L^2(\Omega))$. Furthermore, we will drop the space variable $\bm x$ in functions, e.g. we denote $q(\bm x,t)$ as $q(t)$, when no confusion occurs.

For $\frac{\pi}{2}<\theta<\min\{\pi,\frac{\pi}{\alpha_0}\}$ and $ \delta > 0$, let $\Gamma_\theta$ be the contour in the complex plane defined by
$
\Gamma_\theta := \big \{z\in\C: |{\rm arg}(z)|=\theta, |z|\ge  \delta \big \}
\cup \big \{z \in\C: |{\rm arg}(z)|\le \theta, |z|=  \delta \big \}.
$
For any $w \in L^1_{loc}(\mathcal I)$, the Laplace transform $\mathcal L$ of its extension $\tilde w(t)$ to zero outside $\mathcal I$ and the corresponding inverse transform $\mathcal L^{-1}$ are denoted by
$ \mathcal{L}w(z):=\int_0^\infty \tilde w(t)e^{-tz}d t$ and $ \mathcal{L}^{-1}(\mathcal Lw(z)):=\frac{1}{2\pi \rm i}\int_{\Gamma_\theta} e^{tz}\mathcal Lw(z)d z=w(t)$.
The following inequalities hold for  $0<\mu< 1$ or $\mu=\alpha_0$ and $Q$ independent from $z$ \cite{JinLaz,Lub}
\begin{equation}\label{GammaEstimate}
\int_{\Gamma_\theta} |z|^{\mu-1} |e^{tz}|  \, |d z| \le Q t^{-\mu};~~\|(z^\mu-\kappa A)^{-1}\|\leq Q|z|^{-\mu},~~\forall z\in \Gamma_\theta,
\end{equation}
and we have $\mathcal L(\p_t^\mu w)=z^\mu \mathcal L w-(I_t^{1-\mu}w)(0) $, see \cite{Jinbook}.

Finally we introduce the concept of positive definiteness. A kernel $\beta(t)$ is said to be positive definite if $\int_0^{\bar t}q(t)(\beta*q)(t)dt\geq 0$ for any $ q\in C[0,\bar t] $ and for each $0<\bar t\leq T$,
and $\beta_\mu (t):=\frac{t^{\mu-1}}{\Gamma(\mu)}$ with $0<\mu<1$ is positive definite, see e.g. \cite{Mus}.

Throughout this work, we consider smooth $\alpha(t)$. Specifically, we assume $\alpha(t)$ is three times differentiable with $|\alpha'(t)|+|\alpha''(t)|+|\alpha'''(t)|\leq L$ for some $L>0$. We also denote $\alpha_0=\alpha(0)$,
$I_t^\mu w:=\beta_\mu*w$ and $\p_t^\mu w=\p_t (\beta_{1-\mu}*w)$.

\subsection{Modeling issues} Apart from the applications of (\ref{VtFDEs}) mentioned in the introduction, we present a further physical interpretation of (\ref{VtFDEs}) from the dynamic behavior of an elastic membrane.
Consider an elastic membrane that lies in the $xy$-plane and is bounded by a plane curve $C$ in the plane when in the undeformed state. Let the structure be subjected to distributed load $f_0 (x, y, t)$ in the $z$ direction per unit area. Consider a differential element emanating from the coordinates $(x, y)$ and having sides of length $dx$ and $dy$. Let $w(x, y,t)$ correspond to the transverse displacement of the material particle that was located on the surface $z = 0$ of the membrane at the given coordinates when in the reference configuration. In addition, let $\phi_x(x,y,t)$ represent the rotation of the surface of the membrane whose normal is parallel to the $x$-axis at the given point and let $\phi_y(x,y,t)$ represent the rotation of the surface of the membrane whose normal is parallel to the y-axis. Let $N_x(x, y, t)$ and $N_y(x, y, t)$ correspond to the membrane force per unit length in the indicated direction, $ m_0(x,y)$ is the mass per unit area of membrane. $\p_t w(x,y,t)$ describes the viscous damping force. Applying Newton's second law in the transverse direction gives
\begin{equation}\label{ModelEqn:1}\begin{array}{l}
\ds N_x(x+dx, y, t) \sin \phi_x(x+dx,y,t)dy - N_x(x, y, t) \sin \phi_x(x,y,t)dy \\[0.05in]
\ds \qquad \quad + N_y(x, y+dy, t) \sin \phi_y(x,y+dy,t)dx - N_y(x, y, t) \sin \phi_y(x,y,t)dx \\[0.05in]
\ds \qquad + f_0(x,y,t) dxdy  = m_0(x,y) \p_t^2 w(x,y,t)dxdy + \kappa_0(x,y) \p_t w(x,y,t) dxdy.
\end{array}\end{equation}
Divide equation \eqref{ModelEqn:1} by $dx dy$ and let $dx$ and $dy$ tend to zero to yield
\begin{equation}\label{ModelEqn:2}\begin{array}{l}
  m_0 \p_t^2 w + \kappa_0 \p_t w - \p_x (N_x \sin \phi_x) - \p_y (N_y \sin \phi_y) = f_0.
\end{array}\end{equation}
Since no in-plane external force is applied in the $x$ and $y$ directions, Newton's second law in the $x$ and $y$ directions respectively concludes that $N_x \cos \phi_x = a_x$ and $N_y \cos \phi_y = a_y$ are constants. We use these results and the fact that $\tan \phi_x = \p_x w$ and $\tan \phi_y = \p_y w$ to rewrite equation \eqref{ModelEqn:2} as the following equation
\begin{equation}\label{ModelEqn:3}
  \p_t^2 w + (\kappa_0/m_0) \p_t w - 1/m_0 \p_x ( a_x \p_x w) - 1/m_0 \p_y ( a_y \p_y w) = f_0 /m_0.
\end{equation}

 Numerous commonly used materials, such as polymers, superalloys, and composite materials, exhibit viscoelastic properties. These materials inherit both restorative elastic solid behavior and internally dissipative Newtonian fluid behavior. They possess exhibit power-law memory effects \cite{BonKap,CraRojAti,JaiMcKin,Sol}. However, equation \eqref{ModelEqn:3} was derived based on integer-order rheological models and consequently exhibit exponentially decaying behavior \cite{Pod,Red}. Consequently, the model \eqref{ModelEqn:3} fails to accurately capture the power-law behavior of many modern materials \cite{JaiMcKin,Mai}, as a combination of exponentially decaying functions does not provide an accurate approximation to a power-law function across a broad parameter range.

To address this issue, we introduce the fractional time derivative $^c\p_t^\alpha w$, with $1 < \alpha < 2$, to properly account for the viscous-inertial behavior, which may be viewed as a fractional interpolation of the inertial force and viscous damping force that were modeled by the integer-order inertial term and integer-order viscous damping term in equation \eqref{ModelEqn:3}. In this case, we ended up with the fractional diffusion-wave equation of the form (\ref{VtFDEs}) with $\alpha(t)\equiv \alpha$ \cite{Mai95,Pod}.

Finally, in many applications structures experience long term vibration, which may lead to material fatigue. This process initiates at microscales through the development of microcracks and bond breakage, propagating to macroscales, eventually leading to material fatigue and possible tragedy. Mathematically, this process alters the fractal dimension of the material  \cite{MeeSik,MetKla}, which is closely related to the fractional order $\alpha$. Consequently, the fractional order may change in space and/or time. In this paper we consider the time-dependent case, which leads to the form of (\ref{VtFDEs}).

\subsection{Well-posedness and regularity}
We apply the convolution method proposed in \cite{Zhe} to convert the model (\ref{VtFDEs}) to a more feasible formulation. For this purpose, we apply the variable substitution $z=s/t$ to perform calculations as follows
\begin{align}
(\beta_{\alpha_0-1}*k)(t)&=\int_0^t\frac{(t-s)^{\alpha(0)-2}}{\Gamma(\alpha(0)-1)}\frac{s^{1-\alpha(s)}}{\Gamma(2-\alpha(s))}ds \nonumber\\
&= \int_0^1\frac{(t-tz)^{\alpha(0)-2}}{\Gamma(\alpha(0)-1)}\frac{(tz)^{1-\alpha(tz)}}{\Gamma(2-\alpha(tz))}tdz\nonumber\\
&=\int_0^1\frac{(tz)^{\alpha(0)-\alpha(tz)}}{\Gamma(\alpha(0)-1)\Gamma(2-\alpha(tz))}(1-z)^{\alpha(0)-2}z^{1-\alpha(0)}dz\nonumber \\
& =:g(t),\label{mh7}
\end{align}
where $k$ denoted by \eqref{var_RL} and $g(t)$ is called a generalized identity function.
It is shown in \cite{Zhe} that $g(0)=1$, $|g|\leq Q$ for $t\in [0,T]$ and
$|g'|\leq Q(|\ln t|+1)$.
 Then we calculate the convolution of (\ref{VtFDEs}) and $\beta_{\alpha_0-1}$ as
\begin{align}\label{v6}
\beta_{\alpha_0-1}*[(k*\p_t^2 u)-\kappa A u - f]=0.
\end{align}
We invoke (\ref{mh7}) to obtain
\begin{align}\label{modeln}
g*\p_t^2u -\kappa\beta_{\alpha_0-1}*   A u=\beta_{\alpha_0-1}*f.
\end{align}
An application of the integration by parts leads to the transformed model
\begin{equation}\label{Model2}
\p_t u -\kappa\beta_{\alpha_0-1}*   A u=p_u+g(t)\bar u_0,~~p_u:=\beta_{\alpha_0-1}*f-g'*\p_tu.
\end{equation}
equipped with the initial and boundary conditions (\ref{ibc}).

If we further  integrate (\ref{Model2}) in time, we get
\begin{align}\label{jy1}
   u - u_0 - \kappa \beta_{\alpha_0}* A u = P_u,
\end{align}
where
\begin{align*}
   P_u := \beta_1* p_u + (\beta_1*g ) \bar{u}_0 = \beta_{\alpha_0}* f - g' * (u-u_0) + (\beta_1*g) \bar{u}_0.
\end{align*}

{\color{black}
Based on the expression (\ref{jy1}), $u\in C([0,T];L^2)$ is called a mild solution of model (\ref{VtFDEs})--(\ref{ibc}) if it satisfies (\ref{jy1}). Then we establish the existence and uniqueness of the mild solutions in $C([0,T]; L^2)$ by the Banach fixed-point theorem  in the following theorem, the proof of which is presented in Appendix A.

\begin{theorem}\label{thm:C}
Suppose $f\in C([0,T];L^2)$ and $u_0, \bar{u}_0 \in L^2$, then model (\ref{VtFDEs})--(\ref{ibc}) admits a unique mild solution in $C([0,T];L^2)$ and
\begin{align*}
   \|u\|_{C([0,T];L^2)}  \leq Q\big( \|f\|_{C([0,T];L^2)} + \|u_0\| + \|\bar{u}_0\|  \big).
\end{align*}
\end{theorem}

Then we extend this idea to show the existence and uniqueness of the classical solutions based on the equation (\ref{Model2}), the proof of which can be found in Appendix B.
}



\begin{theorem}\label{thm:well}
Suppose $f\in W^{1,1}(L^2)$ and $u_0,\bar u_0\in \check H^2$, then model (\ref{VtFDEs})--(\ref{ibc}) admits a unique solution in $W^{2,p}(L^2)\cap L^p(\check H^2)$ for $1< p<\frac{1}{2-\alpha_0}$ and
\begin{align*}
&\|u\|_{W^{2,p}(L^2)}+\| u\|_{L^p(\check H^2)}\leq Q\big(\|f\|_{W^{1,1}(L^2)}+\|u_0\|_{\check H^2}+\|\bar u_0\|_{\check H^2}\big).
\end{align*}
\end{theorem}

 From the derivations (\ref{v6})--(\ref{Model2}) we find that a solution in $W^{2,p}(L^2)\cap L^p(\check H^2)$ to (\ref{VtFDEs})--(\ref{ibc}) is also a solution to (\ref{Model2}) and (\ref{ibc}). Conversely, it is shown by the derivations around (\ref{v5}) that a solution in $W^{2,p}(L^2)\cap L^p(\check H^2)$ to (\ref{Model2}) and (\ref{ibc}) is also a solution to (\ref{VtFDEs})--(\ref{ibc}). Thus, the problems (\ref{Model2}) and (\ref{ibc}), and (\ref{VtFDEs})--(\ref{ibc}) are indeed equivalent and it suffices to focus the attention on (\ref{Model2}) and (\ref{ibc}).

 \begin{remark}
 Under suitable smoothness assumptions on data, one could follow \cite[Theorems 3 and 4]{Zhe} to prove the following regularity for  $n=1,2$ and $m=0,1$
 \begin{align}
\|\p_t^{n+1} u\|_{\check H^{2m}}&\leq   Qt^{\alpha_0-(n+1)}.\label{v13}
\end{align}
 \end{remark}

\section{An $\alpha_0$-order scheme} From this section we turn the attention to the numerical approximation. As mentioned at the end of the last section, it suffices to consider the numerical approximation of \eqref{Model2}--(\ref{ibc}).
If we denote $\bar{\alpha} = \alpha_0 -1\in (0,1)$,
 \eqref{Model2} could be rewritten as
\begin{align}
   \partial_t u + \int_0^t g'(t-s) \p_s u(s)ds -  \kappa\int_0^t \beta_{\bar{\alpha}}(t-s)  A u(s)ds = \bar{f}(t), \label{ModelA}
\end{align}
where $\bar{f}(t)=\int_0^t \beta_{\bar{\alpha}}(t-s) f(s)ds + g(t) \bar{u}_0$. If we further define $\tilde{u}(t)=u(t)-u_0$, equation \eqref{ModelA} becomes
\begin{align}
  & \partial_t\tilde{u} + \int_0^t g'(t-s) \p_s \tilde{u}(s)ds -  \kappa\int_0^t \beta_{\bar{\alpha}}(t-s)  A \tilde{u}(s)ds = \mathcal{F}(t), \label{ModelB} \\
  & \tilde{u}(0) =0; ~~\mathcal{F}(t)= \left(  \beta_{\bar{\alpha}}*(\kappa A u_0 + f) \right) (t) + g(t)\bar{u}_0.
   \label{ModelC}
\end{align}

We then propose and analyze an $\alpha_0$-order-in-time scheme for problem \eqref{ModelB}--\eqref{ModelC}.

\subsection{A weak constraint on $\alpha(t)$}
The numerical analysis in the current and the next sections depends on an additional constraint on $\alpha(t)$:
\begin{align}\label{aas}
\alpha'(0)=\alpha''(0)=0.
\end{align}
The reason for imposing this condition is to improve the properties of the generalized identity function $g$, cf. Lemma \ref{lemma5.1}, to facilitate the numerical analysis. It is worth mentioning that this is a very weak constraint since for any smooth $\alpha(t)$, one could find a sequence of smooth functions $\{\alpha_{\sigma}(t)\}_{\sigma> 0}$ satisfying $\alpha_\sigma'(0)=\alpha_\sigma''(0)=0$ for any $\sigma>0$ such that $\max_{t\in [0,T]}|\alpha(t)-\alpha_{\sigma}(t)|\rightarrow 0$ as $\sigma\rightarrow 0$. Consequently, one could replace the smooth $\alpha(t)$ in the model by an approximate function $\tilde \alpha(t)$ (satisfying $\tilde\alpha'(0)=\tilde\alpha''(0)=0$)  whose error could be arbitrarily small such that the constraint of the condition (\ref{aas}) on $\alpha(t)$ could be arbitrarily weak.

To prove this claim, we construct $\alpha_\sigma(t)$ by setting $\alpha_\sigma(t)=\alpha_0$ on $t\in [0,\sigma]$, $\alpha_\sigma(t)=\alpha(t)$ on $t\in [2\sigma,T]$ and $\alpha_\sigma(t)=\alpha_s(t)$ for $t\in[\sigma,2\sigma]$ for some smooth function $1<\alpha_s(t)<2$ satisfying
\begin{align*}&~~\alpha_s(\sigma)=0,
~~\alpha_s(2\sigma)=\alpha(2\sigma);
~~\alpha_s'^{,+}(\sigma)=\alpha_s''^{,+}(\sigma)=\alpha_s'''^{,+}(\sigma)=0;\\
&\alpha_s'^{,-}(2\sigma)=\alpha'^{,+}(2\sigma),~~\alpha_s''^{,-}(2\sigma)=\alpha''^{,+}(2\sigma),~~\alpha_s'''^{,-}(2\sigma)=\alpha'''^{,+}(2\sigma),
 \end{align*}
where $-$ and $+$ imply the left-hand and right-hand derivatives, respectively. By this construction, $\alpha_\sigma(t)$ is a three times continuously differentiable function and satisfies $\alpha_\sigma'(0)=\alpha_\sigma''(0)=0$ and
\begin{align*}
\max_{t\in [0,T]}|\alpha(t)-\alpha_\sigma(t)|\leq \max\{\max_{t\in [0,\sigma]}|\alpha(t)-\alpha_0|,\max_{t\in[\sigma,2\sigma]}|\alpha(t)-\alpha_s(t)|\}.
\end{align*}
By $\alpha(t)-\alpha_0=\alpha'(\xi)t$ we have $|\alpha(t)-\alpha_0|\leq Lt$, while $\alpha(2\sigma)=\alpha_s(2\sigma)$ implies
\begin{align*}
|\alpha(t)-\alpha_s(t)|&=|\alpha(t)-\alpha(2\sigma)+\alpha_s(2\sigma)-\alpha_s(t)| \nonumber \\
&=|\alpha'(\xi_1)(t-2\sigma)+\alpha'_s(\xi_2)(2\sigma-t)|\nonumber \\
&  \leq Q(2\sigma-t),~~t\in [\sigma,2\sigma].
\end{align*}
Combine above estimates to get
$\max\limits_{t\in [0,T]}|\alpha(t)-\alpha_\sigma(t)|\leq Q\sigma\rightarrow 0^+ \mbox{ as }\sigma\rightarrow 0^+$.

\begin{remark}
The condition $\alpha''(0)=0$ (and thus the boundedness of $g''$) is not necessary to prove the sharp bound of the truncation error $(R_2)^n$ appearing in \eqref{eq5.8}. The only reason we set this condition is to avoid technical difficulties in numerical analysis. Indeed, the desired convergence rates of subsequent schemes can be achieved without the condition $\alpha''(0)=0$.
\end{remark}

With the condition (\ref{aas}), we analyze the properties of $g(t)$.
\begin{lemma}\label{lemma5.1}
 Under the condition (\ref{aas}), $|g^{(j)}(t)|\leq Q$ for $t\in [0,T]$ and $j=1,2$ where  $g(t)$ is given by \eqref{mh7}.
\end{lemma}
\begin{proof}
 {\color{black}Let $B(t):=(tz)^{\alpha(0)-\alpha(tz)}$.  As $\Gamma(z)$ is bounded away from $0$ for $z>0$, direct calculations yield
\begin{align*}
     \left| \left(\frac{B(t)}{\Gamma(2-\alpha(tz))} \right)'  \right| & =
   \left| \frac{B'(t)\Gamma(2-\alpha(tz)) - B(t)\Gamma'(2-\alpha(tz))\alpha'(tz)z }{[\Gamma(2-\alpha(tz))]^2}   \right| \\
   & \leq Q \Big( |B'(t)||\Gamma(2-\alpha(tz))| + |B(t)| |\Gamma'(2-\alpha(tz))|  \Big),
\end{align*}
\begin{align*}
     \left| \left(\frac{B(t)}{\Gamma(2-\alpha(tz))} \right)''  \right| &  \leq Q \Big( |B''(t)||\Gamma(2-\alpha(tz))| \\
     & \quad + |B'(t)| |\Gamma'(2-\alpha(tz))| + |B(t)||\Gamma''(2-\alpha(tz))| \Big).
\end{align*}
Thus, according to the definition (\ref{mh7}) of $g$, if we intend to show $|g^{(j)}(t)|\leq Q$,  it suffices to show $|B^{(j)}(t)|\leq Q$ with $j=1,2$.
}

If $\alpha'(0)=0$, the Taylor expansion gives
$$B(t) = e^{-\int_{0}^{tz}(tz-s)\alpha''(s)ds \ln(tz) }$$
such that we have for $\bar{p}=tz$
\begin{equation*}
 \begin{split}
     |B'(t)| 
     & = \Big|B(t) \p_{\bar{p}} \Big( - \int_{0}^{\bar{p}}(\bar{p}-s)\alpha''(s)ds \ln(\bar{p}) \Big) \frac{d\bar{p}}{dt}\Big| \\
     &  = \Big|-B(t) z \Big[  \ln(\bar{p})\int_{0}^{\bar{p}}\alpha''(s)ds + \frac{1}{\bar{p}}\int_{0}^{\bar{p}}(\bar{p}-s)\alpha''(s)ds  \Big]\Big| \\
     & \leq Q(|\bar{p}\ln \bar{p}|+1)\leq Q.
 \end{split}
\end{equation*}
The estimate of $B''$ could be obtained similarly based on $|\ln(\bar{p})\alpha''(\bar{p})|\leq Q$ derived from $\alpha''(0)=0$ and $|\alpha'''|\leq L$.
\end{proof}

\subsection{Construction of $\alpha_0$-order scheme}
Let $N$ be a positive integer and $\tau=T/N$ be the uniform temporal step size with $t_n=n \tau$ for $0\leq n \leq N$. For simplicity, we denote $\tilde{u}^n:=\tilde{u}(t_n)$ and define
\begin{equation*}
  \begin{split}
      &  \delta_t v^{n} = \frac{v^n-v^{n-1}}{\tau}, \quad 1 \leq n \leq N, \quad  \delta_t^{(2)} v^{n} = \frac{3}{2} \delta_t v^{n}-\frac{1}{2} \delta_t v^{n-1}, \quad 2 \leq n \leq N, \\
      & v^{n-1/2} = \frac{v^n + v^{n-1}}{2}, \quad t_{n-1/2} = \frac{t_{n} + t_{n-1}}{2}, \quad 1 \leq n \leq N.
  \end{split}
\end{equation*}

We consider \eqref{ModelB} at  $t=t_n$ for $1\leq n \leq N$
\begin{align}\label{eq5.4}
  & \partial_t\tilde{u}(t_n) + \int_0^{t_n} g'(t_n-s) \p_{s} \tilde{u}(s) ds - \kappa\int_0^{t_n} \beta_{\bar{\alpha}}(t_n-s)  A \tilde{u}(s)ds = \mathcal{F}(t_n).
\end{align}
 We utilize the second-order BDF, see e.g. \cite{Jinsisc,LiaCSIAM}, to discretize the first left-hand side term of \eqref{eq5.4}
\begin{align}\label{eq5.5}
     & \partial_t\tilde{u}(t_1) =  \delta_t \tilde{u}^1 + (R_1)^1, \quad  \partial_t\tilde{u}(t_n) =  \delta_t^{(2)} \tilde{u}^n + (R_1)^n,  \quad 2\leq n \leq N,\\
  & (R_1)^1 = \frac{1}{\tau} \int_{0}^{\tau} t \p_{t}^2 u (t)dt, \label{eq5.6} \\
  & (R_1)^{n} = - \frac{1}{\tau} \int^{t_{n}}_{t_{n-1}} (t-t_{n-1})^2 \p_{t}^3u (t)dt + \frac{1}{4\tau} \int^{t_{n}}_{t_{n-2}} (t-t_{n-2})^2 \p_{t}^3u (t)dt. \label{eq5.7}
\end{align}
For the second left-hand side term of \eqref{eq5.4}, we apply the piecewise linear interpolation
\begin{align}\label{eq5.8}
  & \int_0^{t_n} g'(t_n-s) \p_{s} \tilde{u}(s) ds = \sum\limits_{k=1}^{n} w_{n-k} \frac{\tilde{u}^{k}-\tilde{u}^{k-1}}{\tau} + (R_2)^n, \quad 1\leq n \leq N,
\end{align}
where $w_k = \int_{t_k}^{t_{k+1}}g'(t)dt=g(t_{k+1})-g(t_{k})$ and
\begin{equation*}
  \begin{split}
      (R_2)^n &= \sum\limits_{k=1}^{n} \int_{t_{k-1}}^{t_k} \Big[ \p_t \tilde{u}(t) - \frac{\tilde{u}^k-\tilde{u}^{k-1}}{\tau} \Big] g'(t_n-t) dt \\
     & =  \sum\limits_{k=1}^{n} \int_{t_{k-1}}^{t_k} \Big[  \int_{t_{k-1}}^{t} \p_{s}^2 \tilde{u}(s) \frac{s-t_{k-1}}{\tau} ds -  \int_{t}^{t_{k}} \p_{s}^2 \tilde{u}(s) \frac{t_{k}-s}{\tau}  ds \Big] g'(t_n-t) dt.
  \end{split}
\end{equation*}
By swapping the double integrals we simplify $ (R_2)^n$ as
\begin{align}
    (R_2)^n &= \sum\limits_{k=1}^{n}   \int_{t_{k-1}}^{t_k} \Big[ g(t_n-s) - \Big(  g(t_{n-k}) \frac{s-t_{k-1}}{\tau} + g(t_{n-k+1}) \frac{t_{k}-s}{\tau}  \Big) \Big] \p_{s}^2 \tilde{u}(s)ds \nonumber\\
    & = \sum\limits_{k=1}^{n}    \int_{t_{k-1}}^{t_k} \Big[  \frac{t_{k-1}-s}{\tau} \int_{s}^{t_k} g''(t_n-\theta)(t_k-\theta)d\theta \label{eq5.9}\\
    & \qquad\qquad\qquad\qquad  + \frac{s-t_k}{\tau} \int_{s}^{t_{k-1}} g''(t_n-\theta)(t_{k-1}-\theta)d\theta  \Big] \p_{s}^2 \tilde{u}(s)ds.\nonumber
\end{align}
For the third right-hand side term of \eqref{eq5.4}, we employ the second-order convolution quadrature rule \cite{Lub} to get
\begin{align}\label{eq5.10}
  & \int_0^{t_n} \beta_{\bar{\alpha}}(t_n-s)  A \tilde{u}(s)ds = \tau^{\bar{\alpha}} \Big( \sum\limits_{j=1}^n \chi_{n-j}^{(\bar{\alpha})}  A \tilde{u}^j +  \omega_n^{(\bar{\alpha})}  A \tilde{u}^0 \Big) + (R_3)^n,
\end{align}
where the weights $\chi_{j}^{(\bar{\alpha})} $ can be obtained by the following power series \cite[Lemma 5.3]{Lub}
   \begin{equation*}
   \begin{split}
          \left[ \frac{(1-\zeta)(3-\zeta)}{2}  \right]^{-\bar{\alpha}}=\sum\limits_{j=0}^{\infty}\chi_{j}^{(\bar{\alpha})}\zeta^j,~~ \chi_{n}^{(\bar{\alpha})} = \left(\frac{2}{3} \right)^{\bar{\alpha}}\sum\limits_{j=0}^{n} 3^{-j} c_{n-j}^{(\bar{\alpha})}c_{j}^{(\bar{\alpha})}
           = \mathcal{O}( n^{\bar{\alpha}-1} )
   \end{split}
   \end{equation*}
with $c_{n}^{(\bar{\alpha})}=(-1)^n \binom{-\bar{\alpha}}{n}$. The correction weights are given as $\omega_{n}^{(\bar{\alpha})} = \frac{n^{\bar{\alpha}}}{\Gamma(\bar{\alpha}+1)} - \sum\limits_{j=1}^{n}\chi_{n-j}^{(\bar{\alpha})}$ for $ n\geq 1$.
The quadrature error $(R_3)^n$ in \eqref{eq5.10} could be bounded by \cite[Lemma 7]{Qiu}
    \begin{equation}\label{eq5.13}
    \begin{split}
          |(R_3)^n| \leq  Q \Big[ &  \tau^2t_n^{\bar{\alpha}-1} \left| A \bar{u}_0\right| + \tau^{\bar{\alpha}+1} \int_{t_{n-1}}^{t_n}\left|\p_{s}^2 A u(s)\right|ds  \\
          & + \tau^2 \int_{0}^{t_{n-1}} (t_n-s)^{\bar{\alpha}-1}  \left|\p_{s}^2 A u(s)\right|ds  \Big], \quad n=1,2,\cdots,N.
     \end{split}
     \end{equation}
Now we invoke \eqref{eq5.5}, \eqref{eq5.8} and \eqref{eq5.10} in \eqref{eq5.4} to get
\begin{align}
  &  \delta_t \tilde{u}^1 +  w_{0}  \delta_t \tilde{u}^1 - \kappa \tau^{\bar{\alpha}} \Big( \sum\limits_{j=1}^1 \chi_{1-j}^{(\bar{\alpha})} A \tilde{u}^j +  \omega_1^{(\bar{\alpha})}  A \tilde{u}^0 \Big) = \mathcal{F}(t_1) + R^1,  \label{eq5.14} \\
  &  \delta_t^{(2)} \tilde{u}^n + \sum\limits_{k=1}^{n} w_{n-k}  \delta_t \tilde{u}^k - \kappa \tau^{\bar{\alpha}} \Big( \sum\limits_{j=1}^n \chi_{n-j}^{(\bar{\alpha})}  A \tilde{u}^j +  \omega_n^{(\bar{\alpha})}  A \tilde{u}^0 \Big) = \mathcal{F}(t_n)+R^n, \label{eq5.15}
\end{align}
for $2 \leq n \leq N$ where $R^n = \kappa(R_3)^n-(R_2)^n-(R_1)^n$ and $\tilde{u}^0 = 0$. Then we drop the truncation errors to get the following temporal semi-discrete scheme
\begin{align}
  &  \delta_t \widetilde{U}^1 +  w_{0}  \delta_t \widetilde{U}^1 - \kappa\tau^{\bar{\alpha}} \Big( \sum\limits_{j=1}^1 \chi_{1-j}^{(\bar{\alpha})}  A \widetilde{U}^j +  \omega_1^{(\bar{\alpha})}  A \widetilde{U}^0 \Big) = \mathcal{F}(t_1),  \label{eq5.16} \\
  &  \delta_t^{(2)} \widetilde{U}^n + \sum\limits_{k=1}^{n} w_{n-k}  \delta_t \widetilde{U}^k - \kappa\tau^{\bar{\alpha}} \Big( \sum\limits_{j=1}^n \chi_{n-j}^{(\bar{\alpha})}  A \widetilde{U}^j +  \omega_n^{(\bar{\alpha})}  A \widetilde{U}^0 \Big) = \mathcal{F}(t_n),  \label{eq5.17}
\end{align}
with $2\leq n \leq N$ and {\color{black}$\widetilde{U}^0 = 0$}.  After getting $\widetilde{U}^n$, define $U^n:=\widetilde{U}^n+u_0$ as the temporal semi-discrete numerical solution of \eqref{ModelA}.

\subsection{Analysis of semi-discrete scheme}
We first establish the stability of the semi-discrete scheme.

\begin{theorem} \label{thm5.2} Under the condition (\ref{aas}), the numerical solution $\widetilde{U}^m$ of \eqref{eq5.16}-\eqref{eq5.17} with $1\leq m \leq N$ satisfies
\begin{equation*}
    \|\widetilde{U}^m\| \leq Q \sum\limits_{n=1}^{m} \tau \|\mathcal{F}(t_n)\| \leq Q \Big( \| A u_0\| + \| \bar{u}_0\| + \max\limits_{0\leq t \leq t_m}\|f(t)\| \Big).
\end{equation*}
\end{theorem}
\begin{remark}
By $U^n:=\widetilde{U}^n+u_0$ we immediately obtain the stability of $U^n$.
\end{remark}
\begin{proof} We take the inner product of \eqref{eq5.16}-\eqref{eq5.17} with $\tau \widetilde{U}^1$ and $\tau \widetilde{U}^n$, respectively, and sum for $n$ from 1 to $m$ to get
\begin{equation}\label{eq5.19}
  \begin{split}
     \tau (  \delta_t\widetilde{U}^1, \widetilde{U}^1 ) & + \tau \sum\limits_{n=2}^{m} (  \delta_t^{(2)}\widetilde{U}^n, \widetilde{U}^n ) +
     \tau \sum\limits_{n=1}^{m} \sum\limits_{k=1}^{n} w_{n-k}  (  \delta_t \widetilde{U}^k, \widetilde{U}^n ) \\
     & + \kappa\tau^{1+\bar{\alpha}} \sum\limits_{n=1}^{m} \sum\limits_{j=1}^n  \chi_{n-j}^{(\bar{\alpha})}  ( \nabla \widetilde{U}^j,  \nabla \widetilde{U}^n )
     = \tau \sum\limits_{n=1}^{m} ( \mathcal{F}(t_n),  \widetilde{U}^n ).
  \end{split}
\end{equation}
By \cite[Equation (3.9)]{MclTho} we obtain
\begin{equation}\label{eq5.20}
     \tau (  \delta_t\widetilde{U}^1, \widetilde{U}^1 )  + \tau \sum\limits_{n=2}^{m} (  \delta_t^{(2)}\widetilde{U}^n, \widetilde{U}^n )
      \geq \frac{3}{4}\|\widetilde{U}^m\|^2 - \frac{1}{4}( \|\widetilde{U}^{m-1}\|^2 + \|\widetilde{U}^1\|^2 + \|\widetilde{U}^0\|^2 ).
\end{equation}
Then we apply $\widetilde{U}^0=0$ to obtain
$      \sum\limits_{k=1}^{n} w_{n-k}  \delta_t \widetilde{U}^k = \frac{1}{\tau} \big[ w_0 \widetilde{U}^n - \sum\limits_{k=1}^{n-1}(w_{n-k-1}-w_{n-k})\widetilde{U}^k  \big],
$
which leads to
\begin{align}
     &\Big|\tau \sum\limits_{n=1}^{m} \Big(\sum\limits_{k=1}^{n} w_{n-k}  \delta_t \widetilde{U}^k \Big)  \widetilde{U}^n\Big|  = \Big|\sum\limits_{n=1}^{m} \Big[ w_0 \widetilde{U}^n - \sum\limits_{k=1}^{n-1}(w_{k-1}-w_{k})\widetilde{U}^{n-k}  \Big] \widetilde{U}^n\Big|\nonumber \\
     &\qquad \leq |w_0| \sum\limits_{n=1}^{m} |\widetilde{U}^n|^2 + \sum\limits_{n=2}^{m}\sum\limits_{k=1}^{n-1}|w_{k-1}-w_{k}| |\widetilde{U}^{n-k}| |\widetilde{U}^n|,\nonumber\\
  \label{eq5.21}
     &\Big|\tau \sum\limits_{n=1}^{m}  \sum\limits_{k=1}^{n} w_{n-k}  (  \delta_t \widetilde{U}^k, \widetilde{U}^n )\Big|  \leq Q\Big[ |w_0| \sum\limits_{n=1}^{m} \|\widetilde{U}^n\|^2 \\
     &\qquad\qquad\qquad\qquad\qquad\qquad  + \sum\limits_{n=2}^{m}\sum\limits_{k=1}^{n-1}|w_{k-1}-w_{k}| \|\widetilde{U}^{n-k}\| \|\widetilde{U}^n\| \Big]. \nonumber
  \end{align}
Besides, following from \cite[Lemma 4.1]{Lub2} and $\widetilde{U}^0=0$, we get
\begin{equation}\label{eq5.22}
  \begin{split}
     \sum\limits_{n=1}^{m} \sum\limits_{j=1}^n  \chi_{n-j}^{(\bar{\alpha})}  \big( \nabla \widetilde{U}^j,  \nabla \widetilde{U}^n \big) \geq 0.
  \end{split}
\end{equation}
Based on \eqref{eq5.20}, \eqref{eq5.21} and \eqref{eq5.22},  \eqref{eq5.19} turns into
\begin{equation*}
  \begin{split}
      \frac{3}{4}\|\widetilde{U}^m\|^2 & \leq \frac{1}{4}\Big( \|\widetilde{U}^{m-1}\|^2 + \|\widetilde{U}^1\|^2 \Big) + Q |w_0| \sum\limits_{n=1}^{m} \|\widetilde{U}^n\|^2  \\
      & + Q\sum\limits_{n=2}^{m}\sum\limits_{k=1}^{n-1}|w_{k-1}-w_{k}| \|\widetilde{U}^{n-k} \| \|\widetilde{U}^n\|
      + \tau \sum\limits_{n=1}^{m} \| \mathcal{F}(t_n)\|  \| \widetilde{U}^n \|.
  \end{split}
\end{equation*}
Let $K$ satisfy $\|\widetilde{U}^K\|=\max\limits_{1\leq n \leq m}\|\widetilde{U}^n\|$, then
\begin{equation*}
  \begin{split}
      \frac{1}{4}\|\widetilde{U}^K\|  
      & \leq Q |w_0| \sum\limits_{n=1}^{m} \|\widetilde{U}^n\| + Q\sum\limits_{n=2}^{m}\sum\limits_{k=1}^{n-1}|w_{k-1}-w_{k}|  \|\widetilde{U}^n\|
      + \tau \sum\limits_{n=1}^{m} \| \mathcal{F}(t_n)\|.
  \end{split}
\end{equation*}
Using Lemma \ref{lemma5.1} and the expression of $w_k$, we have
\begin{equation}\label{zq2}
  \begin{split}
      |w_0| \leq \int_0^{\tau} |g'(t)|dt \leq Q \tau,  \quad  \sum\limits_{k=1}^{n-1}|w_{k-1}-w_{k}| \leq \tau \sum\limits_{k=1}^{n-1}\int_{t_{k-1}}^{t_{k+1}} |g''(t)|dt \leq Q\tau.
  \end{split}
\end{equation}
Combine three formulas above to get
$$
     \|\widetilde{U}^m\| \leq \|\widetilde{U}^K\|
       \leq Q \tau \sum\limits_{n=1}^{m} \|\widetilde{U}^n\|
      + 4 \tau \sum\limits_{n=1}^{m} \| \mathcal{F}(t_n)\|
$$ for $1\leq m \leq N$.
The an application of the Gr\"{o}nwall's lemma  completes the proof.
\end{proof}


\begin{theorem}
   Let $U^m=\widetilde{U}^m+u_0$ be the semi-discrete numerical solution of \eqref{ModelA} with $\widetilde{U}^m$ satisfying (\ref{eq5.16})--(\ref{eq5.17}). Then the following error estimate holds under the condition (\ref{aas})
   \begin{equation*}
      \|u(t_m)-U^m\| \leq Q \tau^{\alpha_0},  \quad 1\leq m \leq N.
   \end{equation*}
\end{theorem}
\begin{proof}
We denote
$\widetilde{\rho}^n = \tilde{u}^n - \widetilde{U}^n$ with $0\leq n \leq m \leq N$ with $\widetilde{\rho}^0=0$. Then we subtract \eqref{eq5.16} and \eqref{eq5.17} from \eqref{eq5.14} and \eqref{eq5.15}, respectively, to get the error equations, which have the same form as (\ref{eq5.16})-(\ref{eq5.17}) with $\tilde U^n$ replaced by $\widetilde{\rho}^n$ and $\mathcal F(t_n)$ replaced by $R^n$. Thus, an application of Theorem \ref{thm5.2} yields
$
     \|\widetilde{\rho}^m\| \leq  Q \tau \sum\limits_{n=1}^{m} \left\| \kappa(R_3)^n - (R_2)^n - (R_1)^n \right\|.
$
To bound right-hand side terms, we apply (\ref{v13}), \eqref{eq5.6} and \eqref{eq5.7} to get
\begin{equation}\label{eq5.27}
  \begin{split}
     \tau \sum\limits_{n=1}^{m} \| (R_1)^n \| & = \tau \| (R_1)^1 \| +  \tau \sum\limits_{n=2}^{m} \| (R_1)^n \| \\
     & \leq \int_0^{\tau} t^{\alpha_0-1}dt + Q \tau^2 \int_{t_1}^{t_m} t^{\alpha_0-3}dt + \int_0^{2\tau} t^{\alpha_0-1}dt \leq Q\tau^{\alpha_0}.
  \end{split}
\end{equation}
By \eqref{eq5.9}, Lemma \ref{lemma5.1} and (\ref{v13}), we have
\begin{equation}\label{eq5.28}
  \begin{split}
     \tau \sum\limits_{n=1}^{m} \left\| (R_2)^n \right\|
     & \leq  Q \tau \sum\limits_{n=1}^{m} \tau^2 \int_{0}^{t_n} s^{\alpha_0-2}ds  \leq Q\tau^{2}.
  \end{split}
\end{equation}
In addition, applying \eqref{eq5.13} and (\ref{v13}) yields
\begin{equation}\label{eq5.29}
  \begin{split}
     \tau \sum\limits_{n=1}^{m} \left\| (R_3)^n \right\|
     & \leq   \tau^2 \sum\limits_{n=1}^{m} \tau t_n^{\bar{\alpha}-1} \| A \bar{u}_0\| + \tau^{\bar{\alpha}+2}\int_0^{t_m} s^{\alpha_0-2}ds \\
     & + \tau^3 \sum_{n=2}^{m} \int_0^{t_{n-1}}(t_n-s)^{\bar{\alpha}-1} s^{\alpha_0-2}ds \leq \tau^2 \| A \bar{u}_0\| \int_0^{t_m} s^{\bar{\alpha}-1}ds \\
     &  + Q \tau^{\alpha_0+1} +  Q   \tau^{\bar{\alpha}+1}  \int_0^{t_{m}} s^{\alpha_0-2}ds  \leq Q \tau^{\bar{\alpha}+1}=Q\tau^{\alpha_0}.
  \end{split}
\end{equation}
We combine the above estimates \eqref{eq5.27}--\eqref{eq5.29} and use $\widetilde{\rho}^m =u(t_m)-U^m$ to complete the proof.
\end{proof}

\subsection{Analysis of fully-discrete scheme}\label{sec3.4}

Define a quasi-uniform partition of the domain $\Omega$ with a mesh diameter $h$. Let $S_h\subset H_0^1(\Omega)$ be the set of continuous and piecewise linear functions on $\Omega$ with respect to this partition.
 Let $I$ be the identity operator and define the Ritz projection $\Lambda_h: H_0^1 \rightarrow S_h$ as $
   \left( \nabla(\varphi - \Lambda_h \varphi ), \nabla \chi \right) = 0$ for any $\chi \in S_h$,
with the following approximation property \cite{Tho}
\begin{equation}\label{eq5.31}
        \|(I-\Lambda_h)\varphi\|\leq Qh^2\|\varphi\|_{H^2}, \; \|\p_t(\varphi-\Lambda_h \varphi)\|\leq Qh^2\|\p_t\varphi\|_{H^2}, \; \forall \varphi \in H^2\cap H_0^1.
\end{equation}
Then, for any $\chi_1, \chi_2 \in H_0^1(\Omega)$, the weak formulation of \eqref{eq5.14}-\eqref{eq5.15} reads
\begin{align}
  & (  \delta_t \tilde{u}^1, \chi_1 ) +  w_{0} ( \delta_t \tilde{u}^1, \chi_1 ) + \kappa\tau^{\bar{\alpha}} \chi_{0}^{(\bar{\alpha})} (\nabla \tilde{u}^1, \nabla \chi_1 ) = (\mathcal{F}(t_1) + R^1, \chi_1 ),  \label{eq5.32} \\
  & ( \delta_t^{(2)} \tilde{u}^n, \chi_2)  + \sum\limits_{k=1}^{n} w_{n-k} ( \delta_t \tilde{u}^k, \chi_2) + \kappa\tau^{\bar{\alpha}} \sum\limits_{j=1}^n \chi_{n-j}^{(\bar{\alpha})} (\nabla \tilde{u}^j, \nabla\chi_2) \nonumber \\
   & = (\mathcal{F}(t_n)+R^n, \chi_2), \quad 2\leq n \leq N, \label{eq5.33}
\end{align}
and the fully-discrete scheme aims to find $\widetilde{U}_h^n \in S_h$ with $\widetilde{U}_h^0=0$ such that
\begin{align}
  & (  \delta_t \widetilde{U}_h^1, \chi_1 ) +  w_{0} ( \delta_t \widetilde{U}_h^1, \chi_1 ) + \kappa\tau^{\bar{\alpha}} \chi_{0}^{(\bar{\alpha})} (\nabla \widetilde{U}_h^1, \nabla \chi_1 ) = (\mathcal{F}(t_1), \chi_1 ),  \label{eq5.34} \\
  & ( \delta_t^{(2)} \widetilde{U}_h^n, \chi_2)  + \sum\limits_{k=1}^{n} w_{n-k} ( \delta_t \widetilde{U}_h^k, \chi_2) + \kappa\tau^{\bar{\alpha}} \sum\limits_{j=1}^n \chi_{n-j}^{(\bar{\alpha})} (\nabla \widetilde{U}_h^j, \nabla\chi_2)  = (\mathcal{F}(t_n), \chi_2) \label{eq5.35}
\end{align}
for $2\leq n \leq N$.
After yielding $\widetilde{U}_h^n$, we define the fully-discrete numerical approximation $U_h^n$ to the solution $u(t_n)$ of \eqref{ModelA} as
\begin{equation}\label{eq5.36}
   U_h^n = \widetilde{U}_h^n + \Lambda_h u_0, \quad 0\leq n \leq N.
\end{equation}

Analogous to the proof of Theorem \ref{thm5.2}, the following stability holds for the fully discrete scheme \eqref{eq5.34}-\eqref{eq5.36}
\begin{equation*}
    \|U^m_h\|  \leq Q \Big( \| \Lambda_h u_0 \| + \| A u_{0}\| + \| \bar{u}_{0}\| + \max\limits_{0\leq t \leq t_m}\|f(t)\| \Big), \quad 1\leq m \leq N.
\end{equation*}
Define
\begin{equation}\label{zq3}
\begin{split}
  & u^n - U_h^n = \Lambda_h u^n - U_h^n - (\Lambda_h -I)u^n := \xi^n - \eta^n, \quad 0 \leq n \leq N, \\
  & \tilde{u}^n - \widetilde{U}_h^n = \Lambda_h \tilde{u}^n - \widetilde{U}_h^n - (\Lambda_h -I)\tilde{u}^n := \tilde{\xi}^n - \tilde{\eta}^n, \quad 0 \leq n \leq N,
\end{split}
\end{equation}
where $\xi^0 = \tilde{\xi}^0=0$ and the estimates of $\eta^n, \tilde{\eta}^n$ are known. By $\tilde{u}^n - \widetilde{U}_h^n = (u^n - U_h^n) -  (I-\Lambda_h) u_0$, the estimate of $u^n - U_h^n$ is reduced to that of $\tilde{u}^n - \widetilde{U}_h^n$.

\begin{theorem}\label{thm5.4} Let $U_h^m$ be the fully-discrete approximation of $u(t_m)$ defined in (\ref{eq5.36}) and the condition (\ref{aas}) holds. Then
\begin{equation*}
  \|u(t_m) - U_h^m\| \leq Q \left( h^2 + \tau^{\alpha_0} \right), \quad 1\leq m \leq N.
\end{equation*}
\end{theorem}
\begin{proof} Subtracting \eqref{eq5.34}--\eqref{eq5.35} from \eqref{eq5.32}--\eqref{eq5.33}, respectively, and using Ritz projection we have
\begin{align}
  & (  \delta_t \widetilde{\xi}^1, \widetilde{\xi}^1 ) +  w_{0} ( \delta_t \widetilde{\xi}^1, \widetilde{\xi}^1 ) + \kappa \tau^{\bar{\alpha}}\chi_{0}^{(\bar{\alpha})} (\nabla \widetilde{\xi}^1, \nabla \widetilde{\xi}^1 ) \nonumber \\
   & = (R^1, \widetilde{\xi}^1 ) + (  \delta_t \widetilde{\eta}^1, \widetilde{\xi}^1 ) + w_{0} ( \delta_t \widetilde{\eta}^1, \widetilde{\xi}^1 ),  \nonumber \\
  & ( \delta_t^{(2)} \widetilde{\xi}^n, \widetilde{\xi}^n )  + \sum\limits_{k=1}^{n} w_{n-k} ( \delta_t \widetilde{\xi}^k, \widetilde{\xi}^n ) + \kappa \tau^{\bar{\alpha}} \sum\limits_{j=1}^n \chi_{n-j}^{(\bar{\alpha})} (\nabla \widetilde{\xi}^j, \nabla \widetilde{\xi}^n ) \nonumber \\
   & = (R^n, \widetilde{\xi}^n )+ ( \delta_t^{(2)} \widetilde{\eta}^n, \widetilde{\xi}^n )  + \sum\limits_{k=1}^{n} w_{n-k} ( \delta_t \widetilde{\eta}^k, \widetilde{\xi}^n ), \quad 2\leq n \leq N. \nonumber
\end{align}
Then, summing the above formulas for $n$ from 1 to $m$ and applying \eqref{eq5.22}, we obtain
\begin{equation*}
  \begin{split}
    & \tau (  \delta_t\widetilde{\xi}^1, \widetilde{\xi}^1 )  + \tau \sum\limits_{n=2}^{m} (  \delta_t^{(2)}\widetilde{\xi}^n, \widetilde{\xi}^n ) +
     \tau \sum\limits_{n=1}^{m} \sum\limits_{k=1}^{n} w_{n-k}  (  \delta_t \widetilde{\xi}^k, \widetilde{\xi}^n ) \\
     &
     \leq \tau \sum\limits_{n=1}^{m} ( R^n,  \widetilde{\xi}^n )  +
     \tau \sum\limits_{n=1}^{m} \sum\limits_{k=1}^{n} w_{n-k}  (  \delta_t \widetilde{\eta}^k, \widetilde{\xi}^n ) + \tau ( \delta_t\widetilde{\eta}^1, \widetilde{\xi}^1 )  + \tau \sum\limits_{n=2}^{m} (  \delta_t^{(2)}\widetilde{\eta}^n, \widetilde{\xi}^n ).
  \end{split}
\end{equation*}
Similar to the analysis of Theorem \ref{thm5.2}, we have
\begin{equation}\label{eq5.38}
\begin{split}
     \|\widetilde{\xi}^m\|
       &\leq Q \tau \sum\limits_{n=1}^{m} \|\widetilde{\xi}^n\|
       + \tau \sum\limits_{n=1}^{m} \| R^n \| \\
       & +  \tau \sum\limits_{n=1}^{m} \sum\limits_{k=1}^{n} |w_{n-k}| \| \delta_t \widetilde{\eta}^k\|
    + \tau \| \delta_t\widetilde{\eta}^1\| + \tau \sum\limits_{n=2}^{m} \|  \delta_t^{(2)}\widetilde{\eta}^n \|.
\end{split}
\end{equation}
By $|w_{n-k}|\leq Q\tau$ and \eqref{eq5.31}, we have
\begin{align}
     & \tau \sum\limits_{n=1}^{m} \sum\limits_{k=1}^{n} |w_{n-k}| \| \delta_t \widetilde{\eta}^k\|
       \leq Q \tau \sum\limits_{n=1}^{m} \sum\limits_{k=1}^{n}  \int_{t_{k-1}}^{t_k} \|\p_t \eta(t)\|dt\nonumber  \\
      &\qquad \leq Q \tau \sum\limits_{n=1}^{m} \int_{0}^{t_n} \|\p_t \eta(t)\| dt  \leq Q h^2 \max\limits_{0\leq t \leq t_m} \|\p_t u(t)\|_{H^2},\label{eq5.39}\\
       & \tau \| \delta_t\widetilde{\eta}^1\| \leq \int_0^{\tau} \|\p_t\widetilde{\eta}(t) \|dt \leq  Q \tau h^2 \max\limits_{0\leq t \leq t_1} \|\p_t u(t)\|_{H^2}, \label{eq5.40} \\
       & \tau \sum\limits_{n=2}^{m} \|  \delta_t^{(2)}\widetilde{\eta}^n \| \leq 2\sum\limits_{n=2}^{m} \int_{t_{n-2}}^{t_n}  \|\p_t\widetilde{\eta}(t) \|dt \leq  Q h^2 \max\limits_{0\leq t \leq t_m} \|\p_t u(t) \|_{H^2}. \label{eq5.41}
  \end{align}
We invoke the regularity estimate (\ref{v13}), \eqref{eq5.39}, \eqref{eq5.40} and \eqref{eq5.41} in \eqref{eq5.38} and use the discrete Gr\"{o}nwall's lemma to complete the proof.
\end{proof}

\section{A second-order scheme}

We propose and analyze a second-order-in-time scheme for problem \eqref{ModelA}.

\subsection{Construction of second-order scheme}
Let $
   \widehat{G}(t) = \int_0^{t} g'(t-s) \p_{s} u(s) ds$
and integrate \eqref{ModelA} from $t_{n-1}$ to $t_{n}$ and multiply the resulting equation by $1/\tau$ to get
\begin{equation}\label{GG3}
       \delta_t u^n  + \frac{1}{\tau} \int_{t_{n-1}}^{t_n}\widehat{G}(t)dt -  \frac{\kappa}{\tau} \int_{t_{n-1}}^{t_n} \int_0^t \beta_{\bar{\alpha}}(t-s)  A u(s)ds dt = \frac{1}{\tau} \int_{t_{n-1}}^{t_n} \bar{f}(t) dt =: \bar{F}^n
\end{equation}
for $1\leq n \leq N$.  We combine the middle rectangle formula and \eqref{eq5.8} to discretize the second left-hand side term as
\begin{equation*}
  \begin{split}
       \frac{1}{\tau} \int_{t_{n-1}}^{t_n}\widehat{G}(t)dt & =  \frac{1}{\tau} \Big[ w_0 u^{n-\frac{1}{2}} - \sum\limits_{k=1}^{n-1}(w_{k-1}-w_{k})u^{n-k-\frac{1}{2}} - w_{n-1}u^0  \Big] \\
       &  + (R_{2,*})^{n-\frac{1}{2}} + (R_4)^n,
   \end{split}
\end{equation*}
where $(R_{2,*})^n$ can be denoted by \eqref{eq5.8} ($\tilde{u}$ replaced by $u$) and
\begin{equation}\label{GG5}
  \begin{split}
       (R_4)^n =  \frac{1}{\tau} \int_{t_{n-1}}^{t_n}\widehat{G}(t)dt - \frac{\widehat{G}(t_n)+\widehat{G}(t_{n-1})}{2}.
   \end{split}
\end{equation}
To approximate the convolution term, we apply the averaged PI rule \cite{McLean} to get
\begin{equation*}
  \begin{split}
       \frac{1}{\tau} \int_{t_{n-1}}^{t_n} \int_0^t \beta_{\bar{\alpha}}(t-s)  A u(s)ds dt &= \frac{1}{\tau} \int_{t_{n-1}}^{t_n} \int_0^t \beta_{\bar{\alpha}}(t-s)  A \breve{u}(s)ds dt + (R_5)^n \\
       & := I_{\bar{\alpha}}^{n-1/2}( A u)  + (R_5)^n.
   \end{split}
\end{equation*}
Here
$
       I_{\bar{\alpha}}^{n-1/2}(\varphi) = \omega_{n,1}\varphi^1 + \sum\limits_{j=2}^{n} \omega_{n,j}\varphi^{j-1/2}$
with
$
       \omega_{n,j} = \frac{1}{\tau} \int_{t_{n-1}}^{t_n} \int_{t_{j-1}}^{\min(t,t_j)} \beta_{\bar{\alpha}}(t-s) ds dt >0$ for
       $ 1\leq j \leq n$
and
$
       (R_5)^n =  \frac{1}{\tau} \int_{t_{n-1}}^{t_n} \int_0^t \beta_{\bar{\alpha}}(t-s) \left[  A u(s)-  A \breve{u}(s) \right] ds dt,
$
where $\breve{u}(s) =u^1$ for $ 0<s < \tau$ and $\breve{u}(s) =u^{n-1/2}$ for $ t_{n-1}< s < t_n$ and $ n\geq 2$.
We invoke the above approximations in \eqref{GG3} to get
\begin{equation}\label{GG11}
  \begin{split}
       \delta_t u^n & + \frac{1}{\tau} \Big[ w_0 u^{n-\frac{1}{2}} - \sum\limits_{k=1}^{n-1}(w_{k-1}-w_{k})u^{n-k-\frac{1}{2}} -w_{n-1}u^0  \Big] \\
       & - \kappa I_{\bar{\alpha}}^{n-1/2}( A u)  = \bar{F}^n + (R_*)^n,
  \end{split}
\end{equation}
where $(R_*)^n=\kappa(R_5)^n-(R_4)^n-(R_{2,*})^{n-1/2}$ with $1\leq n \leq N$. Then we drop the truncation errors and replace $u^n$ with its numerical solution $U^n$ to get the following semi-discrete scheme for $1\leq n \leq N$
\begin{align}
   &  \delta_t U^n  + \frac{1}{\tau} \Big[ w_0 U^{n-\frac{1}{2}} - \sum\limits_{k=1}^{n-1}(w_{k-1}-w_{k})U^{n-k-\frac{1}{2}} -w_{n-1}U^0  \Big] - \kappa I_{\bar{\alpha}}^{n-1/2}(A U) \nonumber \\
   &  = \bar{F}^n, \quad  U^0 = u_0.  \label{GG12}
\end{align}

\subsection{Analysis of semi-discrete scheme}
We first establish the following stability result.

{\color{black}
\begin{theorem} \label{thm5.5} Under the condition (\ref{aas}), the numerical solution $U^m$ of \eqref{GG12} satisfies for $1\leq m \leq N$
\begin{equation*}
    \|U^m\| \leq Q\left( \|U^0\| + \tau \sum\limits_{n=1}^{m} \|  \bar{F}^n\| \right) \leq  Q \Big( \|u_0\| + \| \bar{u}_0\| + \max\limits_{0\leq t \leq t_m}\|f(t)\| \Big).
\end{equation*}
\end{theorem}}
\begin{proof}
Taking the inner product of \eqref{GG12} with $\tau U^1$ and $\tau U^{n-1/2}$, respectively, and summing the resulting equations for $n$ from 1 to $m$ we get
\begin{align}
      \tau & (  \delta_tU^1, U^1 ) + \tau \sum\limits_{n=2}^{m} (  \delta_t U^n, U^{n-1/2} ) + \frac{w_0}{2}(U^1-U^0, U^1)  \nonumber\\
     & + \sum\limits_{n=2}^{m} \Big[ w_0 \|U^{n-\frac{1}{2}}\|^2 - \sum\limits_{k=1}^{n-1}(w_{k-1}-w_{k}) ( U^{n-k-\frac{1}{2}}, U^{n-1/2} ) - w_{n-1}(U^0, U^{n-1/2}) \Big] \nonumber\\
     &   + \kappa\tau \Big( I_{\bar{\alpha}}^{1/2}(\nabla U), \nabla U^{1} \Big)  + \kappa\tau \sum\limits_{n=2}^{m}\Big( I_{\bar{\alpha}}^{n-1/2}(\nabla U), \nabla U^{n-1/2} \Big) \nonumber \\
     & = \tau ( \bar{F}^1,  U^1 ) +  \tau \sum\limits_{n=2}^{m} ( \bar{F}^n,  U^{n-1/2} ).\label{eq5.55}
\end{align}
We estimate the first three left-hand side terms as
\begin{equation}\label{eq5.56}
  \begin{split}
     & \tau  (  \delta_t U^1, U^1 ) + \tau \sum\limits_{n=2}^{m} (  \delta_t U^n, U^{n-1/2} ) \geq \frac{\|U^m\|^2-\|U^0\|^2}{2}, \\
     & \left| \frac{w_0}{2}(U^1-U^0, U^1) \right| \leq \frac{1}{2} |w_0| ( \|U^1\|+ \|U^0\|) \|U^1\|.
  \end{split}
\end{equation}
By \cite[pp.~483]{McLean}, the following estimate holds
\begin{equation}\label{eq5.57}
  \begin{split}
      \tau ( I_{\bar{\alpha}}^{1/2}(\nabla U), \nabla U^{1} ) +  \tau \sum\limits_{n=2}^{m}( I_{\bar{\alpha}}^{n-1/2}(\nabla U), \nabla U^{n-1/2} ) \geq 0.
  \end{split}
\end{equation}
We invoke \eqref{eq5.56}--\eqref{eq5.57} in \eqref{eq5.55} and use the Cauchy-Schwarz inequality to get
\begin{equation*}
  \begin{split}
     \|U^m\|^2 &\leq  2\sum\limits_{n=2}^{m} \Big[ |w_0| \|U^{n-\frac{1}{2}}\|^2 + \sum\limits_{k=1}^{n-1}|w_{k-1}-w_{k}| \| U^{n-k-\frac{1}{2}}\|  \|U^{n-1/2} \| \\
     &  + |w_{n-1}| \| U^{0}\|  \|U^{n-1/2} \|  \Big] +  |w_0| (\|U^1\|^2 + \|U^0\|^2)   \\
     &  + \|U^0\|^2 + 2\tau \|  \bar{F}^1\|  \| U^1 \| +  2\tau \sum\limits_{n=2}^{m} \|  \bar{F}^n\|   \|  U^{n-1/2}\|.
  \end{split}
\end{equation*}
Let $K$ be such that $\|U^K\|=\max\limits_{0\leq n \leq m}\|U^n\|$, then
\begin{align*}
      \|U^K\| & \leq 2\sum\limits_{n=2}^{K}\Big[ |w_0|  + \sum\limits_{k=1}^{n-1}|w_{k-1}-w_{k}| +|w_{n-1}|  \Big] \|U^{n-\frac{1}{2}} \| + 2 |w_0| \|U^K\| \\
      & +  \|U^0\| +  2\tau \sum\limits_{n=1}^{K} \|  \bar{F}^n\|.
\end{align*}
By \eqref{zq2} and $|w_{n-1}|\leq \int_{t_{n-1}}^{t_{n}}|g'(s)|ds \leq Q\tau$, we further obtain
\begin{align*}
  \|U^m\| \leq \|U^K\| \leq Q\tau\sum\limits_{n=2}^{m}  \|U^{n-\frac{1}{2}}\| +  \|U^0\| +2\tau \sum\limits_{n=1}^{m} \|  \bar{F}^n\|.
\end{align*}
Then an application of the discrete Gr\"{o}nwall's lemma completes the proof.
\end{proof}

\begin{theorem}\label{thm5.6}
      Let $U^m$ be the semi-discrete numerical solution of \eqref{ModelA} with $U^m$ satisfying (\ref{GG12}). Then under the condition (\ref{aas}), we have
   \begin{equation*}
      \|u(t_m)-U^m\| \leq Q \tau^{2}, \quad 1\leq m \leq N.
   \end{equation*}
\end{theorem}
\begin{proof}
We subtract \eqref{GG12} from \eqref{GG11} to get the following error equations
\begin{equation*}
  \begin{split}
       \delta_t \rho^n & + \frac{1}{\tau} \Big[ w_0 \rho^{n-\frac{1}{2}} - \sum\limits_{k=1}^{n-1}(w_{k-1}-w_{k})\rho^{n-k-\frac{1}{2}}  \Big] - \kappa I_{\bar{\alpha}}^{n-1/2}( A \rho)  =  (R_*)^n
  \end{split}
\end{equation*}
for $1\leq n \leq N$ with $\rho^0=0$ and $\rho^n=u(t_n)-U^n$.
By Theorem \ref{thm5.5}, we have
\begin{equation}\label{eq5.59}
  \begin{split}
      \|\rho^m\|  
      & \leq Q\tau \sum\limits_{n=1}^{m} \| \kappa(R_5)^n-(R_4)^n-(R_{2,*})^{n-1/2} \|.
  \end{split}
\end{equation}
Then we estimate the truncation errors. By \eqref{eq5.28} we have
$
     \tau \sum\limits_{n=1}^{m} \| (R_{2,*})^{n-1/2} \|  \leq Q\tau^{2}.
$
Then, \eqref{GG5} gives
\begin{equation*}
  \begin{split}
       (R_4)^n &=  \frac{1}{\tau} \int_{t_{n-1}}^{t_n}\widehat{G}(t)dt - \widehat{G}(t_{n-1/2}) \\
       & + \Big[ \widehat{G}(t_{n-1/2}) - \frac{\widehat{G}(t_n)+\widehat{G}(t_{n-1})}{2} \Big] := R_{41}^n + R_{42}^n,
   \end{split}
\end{equation*}
which, together with the Taylor expansion, gives
\begin{align*}
  &  R_{41}^n = \frac{1}{2\tau} \int_{t_{n-1}}^{t_{n-1/2}}( t-t_{n-1})^2 \p_t^2 \widehat{G}(t) dt + \frac{1}{2\tau} \int_{t_{n-1/2}}^{t_{n}}( t_{n}-t )^2 \p_t^2 \widehat{G}(t) dt, \\
  &  R_{42}^n = \frac{1}{2} \int_{t_{n-1}}^{t_{n-1/2}}( t_{n-1}-t) \p_t^2 \widehat{G}(t) dt + \frac{1}{2} \int_{t_{n-1/2}}^{t_{n}}( t-t_{n} ) \p_t^2 \widehat{G}(t) dt,
\end{align*}
where
$ \p_t \widehat{G}(t) = g'(0)\p_t u(t) + \int_0^t g''(t-s) \p_{s} u(s) ds$ and $ \p_t^2 \widehat{G}(t) = g'(0)\p_t^2 u(t) +  g''(t)\bar{u}_0 + \int_0^t g''(s) \p_{s}^2 u(t-s) ds. $
Combine the above five equations to get
\begin{equation*}
  \begin{split}
       \tau \sum\limits_{n=1}^{m} \| (R_4)^{n} \|  &\leq  \tau \int_0^{t_{\frac{1}{2}}} t \|\p_t^2 \widehat{G}(t)\|dt + \tau^2 \int_{t_{\frac{1}{2}}}^{t_1} \|\p_t^2 \widehat{G}(t)\|dt \\
       & + 2\sum\limits_{n=2}^{m} \tau^2 \int_{t_{n-1}}^{t_n}  \|\p_t^2 \widehat{G}(t)\|dt,
   \end{split}
\end{equation*}
and we apply Lemma \ref{lemma5.1}, Theorem \ref{thm:well}, and (\ref{v13}) to get
\begin{equation*}
  \begin{split}
       \tau \sum\limits_{n=1}^{m} \left\| (R_4)^{n} \right\| &  \leq Q\Big[\tau \int_0^{t_{1/2}} t^{\alpha_0 -1} dt + \tau^2 \int_{t_{1/2}}^{t_1} t^{\alpha_0 -2} dt  + \tau^2 \sum\limits_{n=2}^{m} \int_{t_{n-1}}^{t_n}   t^{\alpha_0 -2} dt \Big] \\
       & \leq Q\tau^{2}.
   \end{split}
\end{equation*}
To bound $(R_5)^n$, define
$  u^{*}(s) =u(t_1)$ for $ 0<s < \tau$ and $u^{*}(s) =  \tau^{-1}[(t_i-t)u(t_{i-1}) + (t-t_{i-1})u(t_{i})]$ for $t_{i-1}< s < t_i$ and $ i\geq 2$. Then $(R_5)^n=(R_{51})^n+(R_{52})^n$ where
\begin{align*}
   & (R_{51})^n =  \frac{1}{\tau} \int_{t_{n-1}}^{t_n} \int_0^t \beta_{\bar{\alpha}}(t-s) \left[  A u(s)-  A u^{*}(s) \right] ds dt,  \\
   & (R_{52})^n =  \frac{1}{\tau} \int_{t_{n-1}}^{t_n} \int_0^t \beta_{\bar{\alpha}}(t-s) \left[  A u^{*}(s)-  A \breve{u}(s) \right] ds dt.
\end{align*}
Following from \cite[Lemma 3.2 and Eq.~(3.19)]{McLean}, we have
\begin{align*}
   & \tau \sum\limits_{n=1}^{m} \| (R_{51})^{n} \| \leq  Q\Big( \int_0^{t_1} t \| A \p_t u(t)\|dt + \tau^2 \int_{t_1}^{t_m} \| A \p_t^2 u(t)\|dt  \Big),  \\
   & \tau \sum\limits_{n=1}^{m} \| (R_{52})^{n} \| \leq  Q \Big(\tau \int_{t_1}^{t_2} \| A \p_t u(t)\|dt + \tau^2 \int_{t_1}^{t_m} \| A \p_t^2 u(t)\|dt  \Big).
\end{align*}
Therefore, we use (\ref{v13}) to get
$$
       \tau \sum\limits_{n=1}^{m} \| (R_5)^{n} \|  \leq Q\tau \int_0^{2\tau} t^{\alpha_0 -1} dt + Q\tau^2 \int_{t_1}^{t_m}   t^{\alpha_0 -2} dt  \leq Q\tau^{2}.
$$
We finally invoke the above estimates in \eqref{eq5.59} and apply the discrete Gr\"{o}nwall's lemma to complete the proof.
\end{proof}

\subsection{Analysis of fully-discrete scheme} Based on the notations in Section \ref{sec3.4},  the weak formulation of \eqref{GG11} reads for any $\chi_3 \in H_0^1(\Omega)$
\begin{align}
   (  \delta_t u^n, \chi_3 ) & + \frac{1}{\tau} \Big[ w_0 ( u^{n-\frac{1}{2}}, \chi_3 ) - \sum\limits_{k=1}^{n-1}(w_{k-1}-w_{k})( u^{n-k-\frac{1}{2}}, \chi_3 ) - (w_{n-1}u^0, \chi_3 )  \Big] \nonumber \\
  & + \kappa ( I_{\bar{\alpha}}^{n-1/2}(\nabla u), \nabla\chi_3 )  = ( \bar{F}^n, \chi_3 ) + ( (R_*)^n, \chi_3 ),  \label{eq5.63}
\end{align}
and the fully-discrete numerical scheme aims to  find $U_h^n \in S_h$ such that $U_h^0=\Lambda_h u_0$ and
\begin{align}
   (  \delta_t U_h^n, \chi_3 ) & + \frac{1}{\tau} \Big[ w_0 ( U_h^{n-\frac{1}{2}}, \chi_3 ) - \sum\limits_{k=1}^{n-1}(w_{k-1}-w_{k})( U_h^{n-k-\frac{1}{2}}, \chi_3 )  - (w_{n-1}U_h^0, \chi_3 )  \Big] \nonumber \\
  & + \kappa ( I_{\bar{\alpha}}^{n-1/2}(\nabla U_h), \nabla\chi_3 )  = ( \bar{F}^n, \chi_3 ),  \quad \chi_3 \in  S_h. \label{eq5.64}
\end{align}
 Then, the fully-discrete numerical approximation $U_h^n$ to (\ref{ModelA}) is given by \eqref{eq5.64}.
Similar to the proof of Theorem \ref{thm5.5}, we have the stability
\begin{equation*}
    \|U^m_h\|  \leq Q \Big( \| \Lambda_h u_0 \|  + \| \bar{u}_{0}\| + \max\limits_{0\leq t \leq t_m}\|f(t)\| \Big), \quad 1\leq m \leq N.
\end{equation*}
\begin{theorem}\label{thm5.7} Let $U_h^n$ given by \eqref{eq5.64} be the fully-discrete numerical approximation to (\ref{ModelA}). Then the following estimate holds
\begin{equation*}
   \|u(t_m) - U_h^m\| \leq Q \left( h^2 + \tau^{2} \right), \quad 1\leq m \leq N.
\end{equation*}
\end{theorem}
\begin{proof}
Use notations in \eqref{zq3} and the Ritz projection, and subtract \eqref{eq5.64} from \eqref{eq5.63} to get
\begin{align}
   & (  \delta_t \xi^n, \chi_3 )  + \frac{1}{\tau} \Big[ w_0 ( \xi^{n-\frac{1}{2}}, \chi_3 ) - \sum\limits_{k=1}^{n-1}(w_{k-1}-w_{k})( \xi^{n-k-\frac{1}{2}}, \chi_3 )  \Big] \nonumber \\
  & + \kappa ( I_{\bar{\alpha}}^{n-1/2}(\nabla \xi), \nabla\chi_3 )  =  ( (R_*)^n, \chi_3 ) + (  \delta_t \eta^n, \chi_3 ) \nonumber \\
  & + \frac{1}{\tau} \Big[ w_0 ( \eta^{n-\frac{1}{2}}, \chi_3 ) - \sum\limits_{k=1}^{n-1}(w_{k-1}-w_{k})( \eta^{n-k-\frac{1}{2}}, \chi_3 ) - w_{n-1}( \eta^0, \chi_3 )  \Big]. \label{eq5.65}
\end{align}
We choose $\chi_3=\tau \xi^1$ and $\chi_3=\tau \xi^{n-1/2}$ with $2\leq n \leq m$ in \eqref{eq5.65}, respectively, and sum the resulting equations from 1 to $m$ to get
\begin{equation*}
  \begin{split}
      \tau & (  \delta_t \xi^1, \xi^1 ) + \tau \sum\limits_{n=2}^{m} (  \delta_t \xi^n, \xi^{n-1/2} ) + \frac{w_0}{2}\|\xi^1\|^2  \\
     & + \sum\limits_{n=2}^{m} \Big[ w_0 \|\xi^{n-\frac{1}{2}}\|^2 - \sum\limits_{k=1}^{n-1}(w_{k-1}-w_{k}) ( \xi^{n-k-\frac{1}{2}}, \xi^{n-1/2} )  \Big] \\
     & + \kappa\tau ( I_{\bar{\alpha}}^{1/2}(\nabla \xi), \nabla \xi^{1} ) + \kappa \tau \sum\limits_{n=2}^{m}( I_{\bar{\alpha}}^{n-1/2}(\nabla \xi), \nabla \xi^{n-1/2} )  = \tau ( (R_*)^1,  \xi^1 )   \\
     &
      +  \tau \sum\limits_{n=2}^{m} ( (R_*)^n,  \xi^{n-1/2} ) +  \frac{(w_0+ 2) \tau }{2}  ( \delta_t \eta^1,  \xi^1 )  +  \tau \sum\limits_{n=2}^{m} (  \delta_t \eta^n,  \xi^{n-1/2} ) \\
     & + \sum\limits_{n=2}^{m} \Big[ w_0 ( \eta^{n-\frac{1}{2}}, \xi^{n-\frac{1}{2}} ) - \sum\limits_{k=1}^{n-1}(w_{k-1}-w_{k})( \eta^{n-k-\frac{1}{2}}, \xi^{n-\frac{1}{2}} ) - w_{n-1}( \eta^0, \xi^{n-\frac{1}{2}} ) \Big].
  \end{split}
\end{equation*}
By Theorem \ref{thm5.5}, we similarly obtain
\begin{equation}\label{eq5.66}
  \begin{split}
      \|\xi^m\|  &\leq Q\tau \sum\limits_{n=1}^{m} \| (R_*)^n \| + Q\tau \sum\limits_{n=1}^{m} \|  \delta_t \eta^n  \| \\
      &
       + Q\sum\limits_{n=2}^{m} \Big[  |w_0|\|\eta^{n-\frac{1}{2}}\|  + \sum\limits_{k=1}^{n-1}|w_{k-1}-w_{k}|\| \eta^{n-k-\frac{1}{2}}\|   + |w_{n-1}|\|\eta^{0}\|   \Big].
  \end{split}
\end{equation}
By the proof of Theorem \ref{thm5.6},  \eqref{eq5.31}, and \eqref{zq2}, we get
\begin{align}\label{eq5.67}
      &\tau \sum\limits_{n=1}^{m} \| (R_*)^n \| + \tau \sum\limits_{n=1}^{m} \|  \delta_t \eta^n  \| \leq
      Q \tau^2 + Q h^2 \|u\|_{L^1(H^2)},\\
 \label{eq5.68}
     &\sum\limits_{n=2}^{m} \Big[  |w_0|\|\eta^{n-\frac{1}{2}}\|  + \sum\limits_{k=1}^{n-1}|w_{k-1}-w_{k}|\| \eta^{n-k-\frac{1}{2}}\|   + |w_{n-1}|\|\eta^{0}\|   \Big]
       \leq Q h^2 \|u\|_{L^{\infty}(H^2)}.
\end{align}
Invoking \eqref{eq5.67}--\eqref{eq5.68} in \eqref{eq5.66} and using the Gronwall's lemma we have $\|\xi^m \|\leq Q(h^2+\tau^2)$, which, together with $\|\eta^m\|\leq Q h^2 \|u\|_{L^{\infty}(H^2)}$, completes the proof.
\end{proof}

\section{Numerical experiments}\label{sec5} We provide numerical examples to verify the theoretical analysis of two fully discrete schemes.

\subsection{Convergence test}

Let $\Omega=(0,1)$, $\kappa=1$, and $h=\frac{1}{J}$ for some $J\in \mathbb{Z}^+$. We measure the temporal and special convergence rates by ${\rm rate}^t = \log_{2} \big(\frac{E(2\tau,h)}{E(\tau,h)}\big)$ and ${\rm rate}^x = \log_{2} \big(\frac{G(\tau,2h)}{G(\tau,h)}\big)$ with errors
\begin{align*}
   & E(\tau,h) = \Big( h\sum_{j=1}^{J-1} \big| U_{h,j}^N (\tau,h) - U_{h,j}^{2N}(\tau/2,h)\big|^2 \Big)^{\frac{1}{2}}, \\
   & G(\tau,h) = \Big( h\sum_{j=1}^{J-1}  \big| U_{h,j}^N(\tau,h) - U_{h,2j}^{N} (\tau,h/2)\big|^2\Big)^{\frac{1}{2}},
\end{align*}
where the notation $U_{h,j}^N (\tau,h)$ refers to the numerical solution at the last time step computed under the time step size $\tau$ and the spacial mesh size $h$.

\textbf{Example 1: Convergence of $\alpha_0$-order scheme.} Let $T=1/2$, $u_0=\sin(\pi x)$, $\bar{u}_0=\sin(2\pi x)$ and $f=0$. We choose $\alpha(t)=\alpha_0 + \frac{1}{2}t^3$ that satisfies the condition (\ref{aas}). We present errors and convergence rates in Tables \ref{tab1}-- \ref{tab2} under different $\alpha_0$, which justify the $O(\tau^{\alpha_0}+h^2)$ accuracy of the $\alpha_0$-order scheme proved in Theorem \ref{thm5.4}.

\begin{table}
    \center  
    \caption{Errors and convergence rates in time under $J=16$ for Example 1.} \label{tab1}
  \resizebox{\textwidth}{!}{
    \begin{tabular}{cccccccccccc}
      \toprule
    & \multicolumn{2}{c}{$\alpha_0 =1.2$} &
     &\multicolumn{2}{c}{$\alpha_0 =1.5$} &
     &\multicolumn{2}{c}{$\alpha_0 =1.9$}\\
   \cmidrule{2-3}  \cmidrule{5-6} \cmidrule{8-9}
       $N$  & $E(\tau,h)$ & ${\rm rate}^{t}$ & $N$  & $E(\tau,h)$ & ${\rm rate}^{t}$ & $N$  & $E(\tau,h)$ & ${\rm rate}^{t}$\\
      \midrule
        1024   &  $8.7528 \times 10^{-6}$ &  *      &  512    & $3.1318 \times 10^{-5}$  &  *    &  256    & $7.3897 \times 10^{-5}$  &  *    \\
        2048   &  $4.0753 \times 10^{-6}$ &  1.10   &  1024   & $1.0789 \times 10^{-5}$  &  1.54 &  512    & $1.8645 \times 10^{-5}$  &  1.99 \\
        4096   &  $1.8287 \times 10^{-6}$ &  1.16   &  2048   & $3.7405 \times 10^{-6}$  &  1.53 &  1024   & $4.7059 \times 10^{-6}$  &  1.99 \\
        8192   &  $8.3180 \times 10^{-7}$ &  1.14   &  4096   & $1.3035 \times 10^{-6}$  &  1.52 &  2048   & $1.1923 \times 10^{-6}$  &  1.98 \\
        16384  &  $3.5793 \times 10^{-7}$ &  1.22   &  8192   & $4.5598 \times 10^{-7}$  &  1.52 &  4096   & $3.0514 \times 10^{-7}$  &  1.97 \\
      \bottomrule
    \end{tabular}
    }
\end{table}

\begin{table}
    \center 
    \caption{Errors and convergence rates in space under $N=32$  for Example 1.} \label{tab2}
    \resizebox{\textwidth}{!}{
  \begin{tabular}{ccccccccccccccccccc}
      \toprule
    & \multicolumn{2}{c}{$\alpha_0 =1.2$} &
     &\multicolumn{2}{c}{$\alpha_0 =1.5$} &
     &\multicolumn{2}{c}{$\alpha_0 =1.9$}\\
   \cmidrule{2-3}  \cmidrule{5-6} \cmidrule{8-9}
       $J$  & $G(\tau,h)$ & ${\rm rate}^{x}$ & $J$  & $G(\tau,h)$ & ${\rm rate}^{x}$& $J$  & $G(\tau,h)$ & ${\rm rate}^{x}$\\
      \midrule
        32   &  $1.2692 \times 10^{-3}$ &    *    &  32    & $1.0846 \times 10^{-3}$  &    *    &  32    & $9.7202 \times 10^{-4}$  &    *    \\
        64   &  $3.1770 \times 10^{-4}$ &  2.00   &  64    & $2.7120 \times 10^{-4}$  &  2.00   &  64    & $2.4487 \times 10^{-4}$  &  1.99 \\
        128  &  $7.9451 \times 10^{-5}$ &  2.00   &  128   & $6.7802 \times 10^{-5}$  &  2.00   &  128   & $6.1332 \times 10^{-5}$  &  2.00 \\
        256  &  $1.9864 \times 10^{-5}$ &  2.00   &  256   & $1.6951 \times 10^{-5}$  &  2.00   &  256   & $1.5340 \times 10^{-5}$  &  2.00 \\
        512  &  $4.9662 \times 10^{-6}$ &  2.00   &  512   & $4.2376 \times 10^{-6}$  &  2.00   &  512   & $3.8355 \times 10^{-6}$  &  2.00 \\
      \bottomrule
    \end{tabular}
    }
\end{table}

\textbf{Example 2: Convergence of second-order scheme.} Let $T=1$, $u_0=x^4(1-x)^4\in \check{H}^4$, $\bar{u}_0=x^2(1-x)^2\in \check{H}^2$ and $f=0$. We select $\alpha(t)=\alpha_0 + \frac{1}{4}t^3$ such that the condition (\ref{aas}) is satisfied. We present errors and convergence rates in Tables \ref{tab3}--\ref{tab4} with different choices of $\alpha_0$, which indicate the $O(\tau^2+h^2)$ accuracy of the second-order scheme second-order accuracy as proved in Theorem \ref{thm5.7}.

\begin{table}
    \center 
      {\color{black}
    \caption{Errors and convergence rates in time under $J=32$  for Example 2.} \label{tab3}
    \resizebox{\textwidth}{!}{
    \begin{tabular}{cccccccccccc}
      \toprule
    & \multicolumn{2}{c}{$\alpha_0 =1.2$} &
     &\multicolumn{2}{c}{$\alpha_0 =1.4$} &
     &\multicolumn{2}{c}{$\alpha_0 =1.7$}\\
   \cmidrule{2-3}  \cmidrule{5-6} \cmidrule{8-9}
       $N$  & $E(\tau,h)$ & ${\rm rate}^{t}$ & $N$  & $E(\tau,h)$ & ${\rm rate}^{t}$ & $N$  & $E(\tau,h)$ & ${\rm rate}^{t}$\\
      \midrule
        64   &  $4.2949 \times 10^{-7}$ &  *      &  128   & $5.0904 \times 10^{-8}$  &  *    &  256   & $3.1184 \times 10^{-7}$  &  *    \\
        128  &  $1.1333 \times 10^{-7}$ &  1.92   &  256   & $1.3150 \times 10^{-8}$  &  1.95 &  512   & $7.8173 \times 10^{-8}$  &  2.00 \\
        256  &  $3.0359 \times 10^{-8}$ &  1.90   &  512   & $3.3684 \times 10^{-9}$  &  1.97 &  1024  & $1.9567 \times 10^{-8}$  &  2.00 \\
        512  &  $8.3286 \times 10^{-9}$ &  1.87   &  1024  & $8.5272 \times 10^{-10}$  &  1.98 &  2048  & $4.8731 \times 10^{-9}$  &  2.01 \\
        1024 &  $2.1710 \times 10^{-9}$ &  1.94   &  2048  & $2.1176 \times 10^{-10}$ &  2.01 &  4096  & $1.2098 \times 10^{-9}$  &  2.01 \\
      \bottomrule
    \end{tabular}
    }
    }
\end{table}

\begin{table}
    \center 
    {\color{black}
    \caption{Errors and convergence rates in space under $N=32$  for Example 2.} \label{tab4}
    \resizebox{\textwidth}{!}{
     \begin{tabular}{ccccccccccccccccccc}
      \toprule
    & \multicolumn{2}{c}{$\alpha_0 =1.2$} &
     &\multicolumn{2}{c}{$\alpha_0 =1.4$} &
     &\multicolumn{2}{c}{$\alpha_0 =1.7$}\\
   \cmidrule{2-3}  \cmidrule{5-6} \cmidrule{8-9}
       $J$  & $G(\tau,h)$ & ${\rm rate}^{x}$ & $J$  & $G(\tau,h)$ & ${\rm rate}^{x}$& $J$  & $G(\tau,h)$ & ${\rm rate}^{x}$\\
      \midrule
        64   &  $5.6203 \times 10^{-7}$ &  *      &  64    & $4.9671 \times 10^{-7}$  &  *      &  64    & $1.0944 \times 10^{-6}$  &  *    \\
        128  &  $1.4059 \times 10^{-7}$ &  2.00   &  128   & $1.2426 \times 10^{-7}$  &  2.00   &  128   & $2.7367 \times 10^{-7}$  &  2.00 \\
        256  &  $3.5153 \times 10^{-8}$ &  2.00   &  256   & $3.1069 \times 10^{-8}$  &  2.00   &  256   & $6.8420 \times 10^{-8}$  &  2.00 \\
        512  &  $8.7885 \times 10^{-9}$ &  2.00   &  512   & $7.7676 \times 10^{-9}$  &  2.00   &  512   & $1.7105 \times 10^{-8}$  &  2.00 \\
        1024 &  $2.1972 \times 10^{-9}$ &  2.00   &  1024  & $1.9419 \times 10^{-9}$  &  2.00   &  1024  & $4.2763 \times 10^{-9}$  &  2.00 \\
      \bottomrule
    \end{tabular}
    }
    }
\end{table}

\textbf{Example 3: Convergence without condition (\ref{aas}).} Let $T=1$, and $f(x,t)=1$. We select $\alpha(t)=\alpha_0 + \frac{1}{8}\sin(t)$  such that the condition (\ref{aas}) is not satisfied. We then consider smooth and nonsmooth data as follows:

(a) $u_0(x)=\sin(\pi x)$, $\bar{u}_0(x)=x^2(1-x)^2$; \par
{\color{black}(b) $u_0(x)=x^{-1/4}$, $\bar{u}_0(x)=\chi_{(0,\frac{1}{2}]}(x)$  where $\chi$ is the indicator function.}

\vskip 1mm
{\color{black} Tables \ref{tab5} and \ref{tab6} show that two methods still exhibit the desired convergence rates under smooth (case (a)) and nonsmooth (case (b)) initial data, respectively, when the condition (\ref{aas}) is not satisfied.
}

\begin{table}
    \center \footnotesize
    {\color{black}
    \caption{Errors and convergence rates in time under $J=32$  for Example 3 (a).} \label{tab5}
   \begin{tabular}{cccccccccccc}
      \toprule
     &  \multicolumn{4}{c}{$\alpha_0$-order scheme} &
     \multicolumn{2}{c}{Second-order scheme}  \\
   \cmidrule{3-4}  \cmidrule{6-7}
    &   $N$  & $E(\tau,h)$ & ${\rm rate}^{t}$ & $N$  & $E(\tau,h)$ & ${\rm rate}^{t}$  \\
      \midrule
                    &   512   &  $1.0747 \times 10^{-5}$ &  *      &  512   & $8.2491 \times 10^{-7}$  &  *     \\
   $\alpha_0 =1.4$  &   1024  &  $4.3248 \times 10^{-6}$ &  1.31   &  1024  & $1.9335 \times                    10^{-7}$  &  2.09  \\
                    &   2048  &  $1.7029 \times 10^{-6}$ &  1.34   &  2048  & $4.6002 \times 10^{-8}$  &  2.07  \\
                    &   4096  &  $6.6169 \times 10^{-7}$ &  1.36   &  4096  & $1.0997 \times 10^{-8}$  &  2.06  \\
      \midrule
                    &   128   &  $1.0288 \times 10^{-4}$ &  *      &  128   & $3.9443 \times 10^{-5}$  &  *     \\
   $\alpha_0 =1.85$ &   256   &  $2.9044 \times 10^{-5}$ &  1.82   &  256   & $9.6921 \times                    10^{-6}$  &  2.02  \\
                    &   512   &  $8.3513 \times 10^{-6}$ &  1.80   &  512   & $2.4006 \times 10^{-6}$  &  2.01  \\
                    &   1024  &  $2.4114 \times 10^{-6}$ &  1.79   &  1024  & $5.9886 \times 10^{-7}$  &  2.00  \\
      \bottomrule
    \end{tabular}
   }
\end{table}

\begin{table}
    \center \footnotesize
    {\color{black}
    \caption{Errors and convergence rates in time under $J=32$  for Example 3 (b).} \label{tab6}
   \begin{tabular}{cccccccccccc}
      \toprule
     &  \multicolumn{4}{c}{$\alpha_0$-order scheme} &
     \multicolumn{2}{c}{Second-order scheme}  \\
   \cmidrule{3-4}  \cmidrule{6-7}
    &   $N$  & $E(\tau,h)$ & ${\rm rate}^{t}$ & $N$  & $E(\tau,h)$ & ${\rm rate}^{t}$  \\
      \midrule
                    &   512   &  $1.7410 \times 10^{-5}$ &  *      &  512   & $1.2593 \times 10^{-6}$  &  *     \\
   $\alpha_0 =1.4$  &   1024  &  $6.9561 \times 10^{-6}$ &  1.32   &  1024  & $2.9333 \times                    10^{-7}$  &  2.10  \\
                    &   2048  &  $2.7261 \times 10^{-6}$ &  1.35   &  2048  & $6.9473 \times 10^{-8}$  &  2.08  \\
                    &   4096  &  $1.0561 \times 10^{-7}$ &  1.37   &  4096  & $1.6528 \times 10^{-8}$  &  2.07  \\
      \midrule
                    &   128   &  $1.8620 \times 10^{-3}$ &  *      &  128   & $1.1095 \times 10^{-3}$  &  *     \\
   $\alpha_0 =1.85$ &   256   &  $4.6902 \times 10^{-4}$ &  1.99   &  256   & $2.7644 \times                    10^{-4}$  &  2.00  \\
                    &   512   &  $1.1709 \times 10^{-4}$ &  2.00   &  512   & $6.8702 \times 10^{-5}$  &  2.01  \\
                    &   1024  &  $2.9215 \times 10^{-5}$ &  2.00   &  1024  & $1.7106 \times 10^{-5}$  &  2.01  \\
      \bottomrule
    \end{tabular}
}
\end{table}


{\color{black}
\textbf{Example 4: Convergence  for a two-dimensional problem without condition (\ref{aas}).}
 Let $\Omega=(0,1)\times (0,1)$,  $T=1$, $\kappa=1$, and $h_x=h_y=h=\frac{1}{J}$ for some $J\in \mathbb{Z}^+$. We measure the temporal and spatial convergence rates by ${\rm Rate}^t = \log_{2} \big(\frac{E_1(2\tau,h)}{E_1(\tau,h)}\big)$ and ${\rm Rate}^x = \log_{2} \big(\frac{G_1(\tau,2h)}{G_1(\tau,h)}\big)$ with errors computed by
\begin{align*}
   & E_1(\tau,h) = \Big( h^2\sum_{i=1}^{J-1} \sum_{j=1}^{J-1} \big| U_{h,i,j}^N (\tau,h) - U_{h,i,j}^{2N}(\tau/2,h)\big|^2 \Big)^{\frac{1}{2}}, \\
   & G_1(\tau,h) = \Big( h^2\sum_{i=1}^{J-1} \sum_{j=1}^{J-1}  \big| U_{h,i,j}^N(\tau,h) - U_{h,2i,2j}^{N} (\tau,h/2)\big|^2\Big)^{\frac{1}{2}}.
\end{align*}
Here the notation $U_{h,i,j}^N (\tau,h)$ refers to the numerical solution at the last time step computed under the time step size $\tau$ and the spacial mesh size $h$.
Let $u_0(x,y)=\sin(\pi x)\sin(\pi y)$, $\bar{u}_0(x,y)=x^2(1-x)^2y^2(1-y)^2$ and $f(x,y,t)=1$. We select $\alpha(t)=\alpha_0 + \frac{1}{9}\sin(t)$  such that the condition (\ref{aas}) is not satisfied.

We test the errors and time-space convergence rates in Tables \ref{tab7} and \ref{tab8} for $\alpha_0$-order and second-order schemes, respectively, from which we observe that the numerical results of two schemes are consistent with the theory.

\begin{table}
    \center \footnotesize
    {\color{black}
    \caption{Errors and convergence rates of $\alpha_0$-order scheme for Example 4.} \label{tab7}
   \begin{tabular}{cccccccccccc}
      \toprule
     &  \multicolumn{4}{c}{$J=32$} &
     \multicolumn{2}{c}{$N=32$}  \\
   \cmidrule{3-4}  \cmidrule{6-7}
    &   $N$  & $E_1(\tau,h)$ & ${\rm Rate}^{t}$ & $J$  & $G_1(\tau,h)$ & ${\rm Rate}^{x}$  \\
      \midrule
                    &   256   &  $1.2947 \times 10^{-5}$ &  *      &  16   & $3.7105 \times 10^{-3}$  &  *     \\
   $\alpha_0 =1.2$  &   512   &  $5.8008 \times 10^{-6}$ &  1.16   &  32   & $9.3281 \times                     10^{-4}$  &  1.99  \\
                    &   1024  &  $2.6169 \times 10^{-6}$ &  1.15   &  64   & $2.3353 \times 10^{-4}$  &  2.00  \\
                    &   2048  &  $1.1246 \times 10^{-6}$ &  1.22   &  128  & $5.8402 \times 10^{-5}$  &  2.00  \\
      \midrule
                    &   64    &  $1.5684 \times 10^{-3}$ &  *      &  16   & $9.9695 \times 10^{-3}$  &  *     \\
   $\alpha_0 =1.9$  &   128   &  $4.0947 \times 10^{-4}$ &  1.94   &  32   & $2.5074 \times                     10^{-3}$  &  1.99  \\
                    &   256   &  $1.0407 \times 10^{-4}$ &  1.98   &  64   & $6.2777 \times 10^{-4}$  &  2.00  \\
                    &   512   &  $2.6242 \times 10^{-5}$ &  1.99   &  128  & $1.5700 \times 10^{-4}$  &  2.00  \\
      \bottomrule
    \end{tabular}
}
\end{table}

\begin{table}
    \center \footnotesize
    {\color{black}
    \caption{Errors and convergence rates of second-order scheme for Example 4.} \label{tab8}
   \begin{tabular}{cccccccccccc}
      \toprule
     &  \multicolumn{4}{c}{$J=32$} &
     \multicolumn{2}{c}{$N=32$}  \\
   \cmidrule{3-4}  \cmidrule{6-7}
    &   $N$  & $E_1(\tau,h)$ & ${\rm Rate}^{t}$ & $J$  & $G_1(\tau,h)$ & ${\rm Rate}^{x}$  \\
      \midrule
                    &   256   &  $2.9708 \times 10^{-6}$ &  *      &  16   & $7.9603 \times 10^{-5}$  &  *     \\
   $\alpha_0 =1.4$  &   512   &  $7.5480 \times 10^{-7}$ &  1.98   &  32   & $2.0688 \times                     10^{-5}$  &  1.94  \\
                    &   1024  &  $1.9065 \times 10^{-7}$ &  1.99   &  64   & $5.2216 \times 10^{-6}$  &  1.99  \\
                    &   2048  &  $4.8052 \times 10^{-8}$ &  1.99   &  128  & $1.3085 \times 10^{-6}$  &  2.00  \\
      \midrule
                    &   64    &  $4.7253 \times 10^{-4}$ &  *      &  16   & $4.2162 \times 10^{-3}$  &  *     \\
   $\alpha_0 =1.85$ &   128   &  $1.1403 \times 10^{-4}$ &  2.05   &  32   & $1.0539 \times                     10^{-3}$  &  2.00  \\
                    &   256   &  $2.7933 \times 10^{-5}$ &  2.03   &  64   & $2.6348 \times 10^{-4}$  &  2.00  \\
                    &   512   &  $6.9038 \times 10^{-6}$ &  2.02   &  128  & $6.5870 \times 10^{-5}$  &  2.00  \\
      \bottomrule
    \end{tabular}
}
\end{table}

}

\subsection{Modeling transition behavior by variable exponent}
We investigate the behavior of the solutions to problem \eqref{VtFDEs}--(\ref{ibc}) in order to show that the variable exponent could accommodate the transition between different vibration states caused by, e.g. the change of the properties of the complex materials with respect to time. Let $T=15$, $\Omega=(0,2)$, $\mu=0.1$, $u_0=\Bar{u}_0=0$ and $f=e^{-(t+(x-1)^2/2)}$, which is a smooth approximation of a point source located at $(x,t)=(1,0)$. We also define a parameterized function with given parameters $z_1$ and $z_2$
$$a(t;z_1,z_2):=z_1+(z_2-z_1)\Big(1-\frac{t}{T}-\frac{\sin(2\pi(1-t/T))}{2\pi} \Big), $$
which is a smooth and monotonic function with rspect to $t$ and satisfies $a(0;z_1,z_2)=z_2$ and $a(T;z_1,z_2)=z_1$. The $\alpha_0$-order scheme is applied with $J=128$ and $N=512$.

 Fig.~\ref{Fig1} contains two solution curves under constant exponents $1.9$ and $1.4$ (which correspond to the (nearly) wave and diffusion-wave equations, respectively) and one solution curve under a variable exponent that transits from $1.9$ to $1.4$ on $t\in [0,1]$ and stays $1.4$ on $t\in [1,T]$. We see that the solution curve under variable exponent behaves like that under $\alpha(t)=1.9$ near the initial time, which then gets close to the solution curve under $\alpha(t)=1.4$. Such transition between wave and diffusion-wave behaviors indicates the modeling capability of the variable exponent in describing mechanical behaviors in complex viscoelastic materials with varying material properties.

\begin{figure}
\setlength{\abovecaptionskip}{0pt}
\centering
\includegraphics[width=0.7\textwidth]{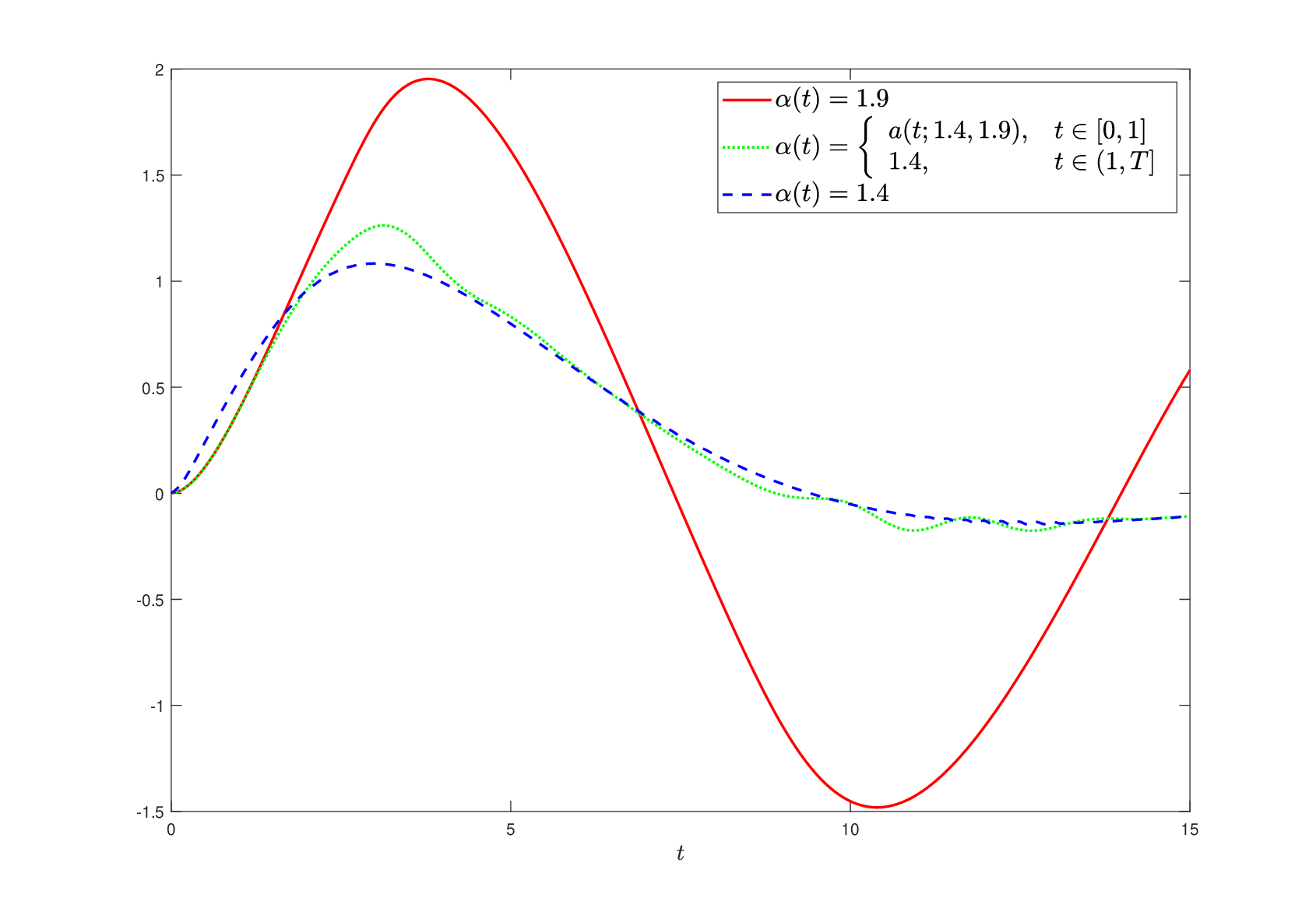}
\caption{Plots of the solutions at $x=1$ under different selections of $\alpha(t)$.}
\label{Fig1}
\end{figure}

\section{Concluding remarks}
This work studies two high-order schemes for (\ref{VtFDEs})--(\ref{ibc}) by overcoming the difficulties caused by the ``indefinite-sign, nonpositive-definite, nonmonotonic'' convolution kernel in the transformed model, which resolves the intricate issue introduced by the model transformation. Therefore, the developed ideas and methods in this work could be combined with the model transformation methods in \cite{Zhe} to design high-order schemes for other nonlocal problems where the variable-exponent terms serve as leading terms, such as the variable-exponent subdiffusion  in anomalous diffusion \cite{Sun} and the variable-exponent integro-differential equations arising from viscoelastic materials \cite{Dan}. 

{\color{black}There are possible extensions of the current work in various aspects. For instance, Theorem 1 proves the well-posedness of the proposed model under the condition $u_0, \Bar{u}_0 \in \check H^2$, which seems to be a strong requirement. Such requirement arises from the application of the solution representation method, see also the first work that adopts this method in variable-exponent fractional problems \cite{Li3}. Thus, it is difficult to relax the regularity of the initial conditions under the current method. We will address this issue by attempting different analysis methods, e.g., express the solution operators by Mittag-Leffler functions for more subtle estimates, in the near future.

Another potential improvement is to develop the numerical scheme under nonsmooth data. There are some sophisticated investigations on this topic for the subdiffusion based on the solution representation technique \cite{Jin,Li}. For model (\ref{VtFDEs})--(\ref{ibc}) proposed in this work, the solution representation still involves a complex term containing the solutions, which increases the difficulty of designing numerical schemes and performing the error analysis. Nevertheless, numerical experiments for problem \eqref{VtFDEs}--\eqref{ibc} with the nonsmooth data indicate that the proposed numerical scheme still exhibits the desired convergence rates. We will consider analyzing error estimates under nonsmooth data in future works.
}

\section*{Appendix A: Proof of Theorem \ref{thm:C}}
{\color{black}
Define the equivalent norm of the space $C([0,T];L^2)$ by
$$\|q\|_{\tilde C([0,T];L^2)} :=\|e^{-\sigma t}q\|_{C([0,T];L^2)}$$ for some $\sigma\geq 1$, and define a mapping $\mathcal N: C([0,T];L^2) \rightarrow C([0,T];L^2)$ by $w=\mathcal N v$, where $w$ satisfies
 \begin{align}\label{ww}
   w - u_0 - \kappa \beta_{\alpha_0}* A w  = P_v, \quad (\bm x,t)\in\Omega\times (0,T],
\end{align}
equipped with zero boundary conditions. We apply the Laplace transform to \eqref{ww} to obtain
$\mathcal L w - u_0 z^{-1} - z^{-\alpha_0}\kappa A \mathcal L w = \mathcal L P_v$. Thus we have $\mathcal Lw=z^{\alpha_0}(z^{\alpha_0}-\kappa A)^{-1}(z^{-1}u_0 +\mathcal LP_v). $
We then take the inverse Laplace transform to obtain
\begin{align}\label{ww1}
  w = F(t) u_0  + \int_0^t F'(t-s)P_v(s)ds,
\end{align}
where the operator $F(\cdot)$ is defined as
$F(t)w:=\frac{1}{2\pi \rm i}\int_{\Gamma_\theta} e^{tz}z^{\alpha_0-1}(z^{\alpha_0}-\kappa A)^{-1}wd z $. We use the Laplace transform to get
\begin{align*}
  \mathcal L \bigg [ \int_0^{t}  F'(t-s) P_v(s)  d s\bigg ]
= \big(z^{\alpha_0-\varepsilon}(z^{\alpha_0}-\kappa A)^{-1}\big)\big(z^\varepsilon \mathcal L P_v \big).
\end{align*}
Then we use the inverse transform to get $\int_0^{t}  F'(t-s) P_v(s)  d s  = \int_0^t \mathcal R(t-s) \partial_s^\varepsilon P_v(s)d s
$, where $\mathcal R(t):=\frac{1}{2\pi \rm i}\int_{\Gamma_\theta}
z^{\alpha_0-\varepsilon} (z^{\alpha_0}-\kappa A)^{-1} e^{zt} d z$. Use (\ref{GammaEstimate}) to bound $\mathcal R$ by
\begin{align}
\|\mathcal R(t-s)\| &\le Q\int_{\Gamma_\theta}
|z|^{\alpha_0-\varepsilon} \|(z^{\alpha_0}-\kappa A)^{-1}\| |e^{z(t-s)}| |d z|\nonumber\\
&\leq Q\int_{\Gamma_\theta}
|z|^{\alpha_0-\varepsilon} |z|^{-\alpha_0} |e^{z(t-s)}| |d z|\leq  \f{Q}{(t-s)^{1-\varepsilon}}.\label{Well:e4}
\end{align}
 We use the definition of $P_v$ to obtain
\begin{align*}
   \partial_t^\varepsilon P_v  & = \p_t \Big[ \beta_{1-\varepsilon} * \beta_{\alpha_0} *f \Big] - \p_t \Big[  (\beta_{1-\varepsilon} * g')* (v-u_0) \Big] + \p_t \Big( \beta_{1-\varepsilon} * \beta_{1} *g \Big) \bar{u}_0  \\
   & = \p_t \Big[ \beta_{1+\alpha_0-\varepsilon} *f \Big] -   \Big[\p_t(\beta_{1-\varepsilon} * g') \Big]  * (v-u_0) + \p_t \Big( \beta_{2-\varepsilon} *g \Big)  \bar{u}_0
\\
   & = \beta_{\alpha_0-\varepsilon} * f  -   \Big[\p_t(\beta_{1-\varepsilon} * g') \Big]  * (v-u_0) + \big( \beta_{1-\varepsilon} *g \big)  \bar{u}_0.
\end{align*}
By $|g'|\leq Q(1+|\ln t|)$ and $|g''|\leq Qt^{-1}$ \cite[Equations (6)--(7)]{Zhe}, we have
 \begin{align}
|\beta_{1-\varepsilon}*g'|& = \Big|\int_0^t\frac{(t-s)^{-\varepsilon}}{\Gamma(1-\varepsilon)}g'(s)ds\Big| = \Big|t^{1-\varepsilon} \int_0^t\frac{(1-s/t)^{-\varepsilon}}{\Gamma(1-\varepsilon)}g'(s)d(s/t)\Big|  \nonumber \\
& =\Big|t^{1-\varepsilon}\int_0^1\frac{(1-z)^{-\varepsilon}}{\Gamma(1-\varepsilon)}g'(tz)dz\Big|\nonumber  \\
&\leq Qt^{1-\varepsilon}\int_0^1(1-z)^{-\varepsilon}\frac{(tz)^{\varepsilon}(1+|\ln(tz)|)}{(tz)^{\varepsilon}}dz \leq Qt^{1-2\varepsilon},\nonumber \\
|\p_t(\beta_{1-\varepsilon}*g')|&=\Big|(1-\varepsilon)t^{-\varepsilon}\int_0^1\frac{(1-z)^{-\varepsilon}}{\Gamma(1-\varepsilon)}g'(tz)dz\nonumber\\
&\qquad+t^{1-\varepsilon}\int_0^1\frac{(1-z)^{-\varepsilon}}{\Gamma(1-\varepsilon)}g''(tz)zdz\Big|\leq Qt^{-2\varepsilon}.\label{mh4}
\end{align}
We combine the above estimates and $|g|\leq Q$ to get
\begin{align*}
   \Big\|   \int_{0}^{t}F'(t-s)P_v(s)ds  \Big\|  & \leq \int_{0}^{t}\|\mathcal{R}(t-s)\| \big\|\partial_s^\varepsilon P_v(s)\big\|ds \\
   &\hspace{-0.3in} \leq Q t^{\varepsilon -1} * \Big[ \beta_{\alpha_0 - \varepsilon} * \|f\| + t^{-2\varepsilon}* (\|v\|+\|u_0\|) + \beta_{2 - \varepsilon} \|\bar{u}_0\|  \Big] \\
   &\hspace{-0.3in} \leq  Q \beta_{\alpha_0}* \|f\| + Q \beta_{1-\varepsilon} * (\|v\|+\|u_0\|) + Q \|\bar{u}_0\|,
\end{align*}
which, together with
$$e^{-\sigma t} \int_{0}^{t}\beta_{1-\varepsilon}(s)ds = \int_{0}^{t} e^{-\sigma s}\beta_{1-\varepsilon}(s) e^{-\sigma (t-s)}ds \leq \int_{0}^{\infty} e^{-\sigma s}\beta_{1-\varepsilon}(s) ds \leq Q \sigma^{\varepsilon-1},$$
leads to
\begin{align*}
  \Big\|  e^{-\sigma t} & \Big\|  \int_{0}^{t}F'(t-s)P_v(s)ds \Big\|   \Big\|_{C([0,T];L^2)}
    \\
    &  \leq  Q \|f\|_{C([0,T];L^2)}  + Q \|u_0\| + Q \|\bar{u}_0\|  +Q \sigma^{\varepsilon-1} \|v\|_{\tilde{C}([0,T];L^2)}.
\end{align*}
In addition, we apply (\ref{GammaEstimate}) to obtain
\begin{align*}
   \|F(t)\| & \leq Q\int_{\Gamma_\theta}
|z|^{\alpha_0-1} \left\|(z^{\alpha_0}-\kappa A)^{-1}\right\| |e^{zt}| |d z|\nonumber\\
& \leq Q \int_{\Gamma_\theta}
|z|^{\alpha_0-1} |z|^{-\alpha_0} |e^{zt}| |d z|\leq  Q \int_{\Gamma_\theta}
|z|^{-1} |e^{zt}| |dz|.
\end{align*}
We consider the contour $\Gamma_{\theta} = \{re^{i\theta} : r \geq \delta\} \cup \{re^{-i\theta} : r \geq \delta\} \cup \{\delta e^{i\psi} : |\psi| \leq \theta\}$ with parameters $\delta = 1/t$ and $\theta \in (\pi/2, \pi)$. For the ray components $z = re^{\pm i\theta}$ where $r \geq \delta$, we have $|dz| = dr$ and $|e^{zt}|/|z| = r^{-1}e^{rt\cos\theta}$. Since $\cos\theta < 0$, the integral over both rays could be bounded via the transformation $q= -r t \cos\theta$
\begin{align*}
  2\int_{1/t}^\infty r^{-1}e^{rt\cos\theta}dr
&= 2\int_{-\cos\theta}^{\infty} \frac{e^{-q}}{q} dq   \leq \frac{2}{|\cos\theta|}\int_{-\cos\theta}^{\infty} e^{-q} dq=  \frac{2e^{\cos\theta}}{|\cos\theta|}.
\end{align*}
For the circular arc $z = \delta e^{i\psi}$ with $|\psi| \leq \theta$, we have $|dz| = \delta d\psi$ and $|e^{zt}|/|z| = \delta^{-1}e^{t\delta\cos\psi} = te^{\cos\psi}$ such that
\begin{align*}
\int_{-\theta}^\theta te^{\cos\psi} \cdot \frac{1}{t}d\psi = \int_{-\theta}^\theta e^{\cos\psi}d\psi \leq 2\theta e.
\end{align*}
We combine the above estimates to get $\|F(t)u_0\|\leq Q\|u_0\|$ and thus
 \begin{equation}\label{ww2}
    \|w\|_{\tilde{C}([0,T];L^2)}  \leq Q\big( \|u_0\| + \|\bar{u}_0\| + \|f\|_{C([0,T];L^2)} + \sigma^{\varepsilon-1} \|v\|_{\tilde{C}([0,T];L^2)} \big),
 \end{equation}
 which shows that the mapping $\mathcal N$ is well defined. To show the contractivity of $\mathcal N$, let $w_1=\mathcal N v_1$ and $w_2=\mathcal N v_2$ such that $w_1-w_2$ satisfies \eqref{ww} with $v=v_1-v_2$, $u_0=\bar{u}_0\equiv 0$ and $f\equiv 0$. Hence \eqref{ww2} shows $\|w_1-w_2\|_{\tilde C ([0,T];L^2)}\leq Q\sigma^{\varepsilon-1}\|v_1-v_2\|_{\tilde C ([0,T];L^2)}$. Select $\sigma$ large enough satisfying that $Q\sigma^{\varepsilon-1}<1$ we conclude that $\mathcal N$ is a contraction mapping, i.e. there exists a unique mild solution $u\in  C ([0,T];L^2)$ for (\ref{jy1}). With $v=w=u$, large $\sigma$ and the equivalence between $\|\cdot\|_{C ([0,T];L^2)}$ and $\|\cdot\|_{\tilde C ([0,T];L^2)}$, we get the stability estimate by \eqref{ww2}
  \begin{equation}\label{ww3}
    \|u\|_{C([0,T];L^2)}  \leq Q\big( \|u_0\|  + \|\bar{u}_0\| + \|f\|_{C([0,T];L^2)} \big).
 \end{equation}
 The uniqueness of the mild solutions to (\ref{jy1}) follows from the stability estimate, which completes the proof.
}

\section*{Appendix B: Proof of Theorem \ref{thm:well}}
We first consider model (\ref{Model2}) and (\ref{ibc}) with $u_0=\bar u_0\equiv 0$, and define the space $\tilde W^{2,p}(L^2):=\{q\in W^{2,p}(L^2):q(0)=\p_tq(0)=0\}$ equipped with the norm $\|q\|_{\tilde W^{2,p}(L^2)}:=\|e^{-\sigma t}\p_t^2q\|_{L^p(L^2)}$ for some $\sigma\geq 1$, which is equivalent to the standard norm $\|q\|_{W^{2,p}(L^2)}$ for $q\in \tilde W^{2,p}(L^2)$. Define a mapping $\mathcal M:\tilde W^{2,p}(L^2)\rightarrow \tilde W^{2,p}(L^2)$ by $w=\mathcal M v$ where $w$ satisfies
\begin{equation}\label{Modelw}
\p_tw-\kappa\beta_{\alpha_0-1} *  A w=p_v,~~(\bm x,t)\in\Omega\times (0,T],
\end{equation}
equipped with zero initial and boundary conditions. By similar derivations among (\ref{ww})--(\ref{ww1}), $w$ could be expressed as
$
w=\int_0^tF(t-s)p_v(s)ds.
$
To show the well-posedness of $\mathcal M$, we need to show $w\in  \tilde W^{2,p}(L^2)$. Differentiate $
w=\int_0^tF(t-s)p_v(s)ds
$ twice in time to obtain
\begin{equation}\label{mh9}
\p_t^2w=\p_tp_v+\int_0^tF'(t-s)\p_sp_v(s)ds.
\end{equation}
Direct calculation yields
\begin{align}
\p_tp_v&=\beta_{\alpha_0-1}f(0)+\beta_{\alpha_0-1}*\p_tf-g'*\p_t^2v,\label{v1}\\
\p_t^\ve \p_tp_v&=\beta_{\alpha_0-\ve-1}f(0)+\beta_{\alpha_0-\ve-1}*\p_tf-[\p_t(\beta_{1-\ve}*g')]*\p_t^2v.\label{v3}
\end{align}

Then we apply (\ref{v1}), Young's convolution inequality, $C([0,T];L^2)\subset W^{1,1}(L^2)$ \cite{AdaFou}, $\beta_{\alpha_0-1}\in L^p(0,T)$ and $v\in \tilde W^{2,p}(L^2)$ to bound $\p_tp_v$ as
\begin{align}\label{vmh9}
\|\p_tp_v\|_{{\color{black}L^{p}}(L^2)}&\leq Q\|f\|_{W^{1,1}(L^2)}+Q\|\beta_{\alpha_0-1}*\|\p_tf\|\|_{L^p(0,T)}\nonumber\\
&\qquad\quad+Q {\color{black} \Big\| e^{-\sigma t} \int_{0}^{t} g'(t-s) \p_t^2v(s) ds \Big\|_{L^p(0,T)} } \nonumber\\
&= Q\|f\|_{W^{1,1}(L^2)}+Q\|\beta_{\alpha_0-1}*\|\p_tf\|\|_{L^p(0,T)}\nonumber\\
&\qquad\quad+Q {\color{black}\Big\| \int_{0}^{t} e^{-\sigma (t-s)} g'(t-s)  e^{-\sigma s} \p_t^2v(s) ds \Big\|_{L^p(0,T)} } \nonumber\\
&\leq Q\|f\|_{W^{1,1}(L^2)}+Q\|\beta_{\alpha_0-1}*\|\p_tf\|\|_{L^p(0,T)}\\
&\qquad\qquad+Q \big\| |e^{-{\color{black}\sigma} t}g'|* |e^{-{\color{black}\sigma} t}\p_t^2v| \big\|_{L^p(0,T)}\nonumber\\
&\leq Q\|f\|_{W^{1,1}(L^2)}+Q\|e^{-{\color{black}\sigma} t}t^{-\ve}\|_{L^1(0,T)}\|v\|_{\tilde W^{2,p}(L^2)}\nonumber\\
&\leq Q\|f\|_{W^{1,1}(L^2)}+Q\sigma^{\ve-1}\|v\|_{\tilde W^{2,p}(L^2)}, \nonumber
\end{align}
 where we used the fact that $|g'|\leq Q(1+|\ln t|)\leq Q t^{-\varepsilon}$ and
$
 \|e^{-\sigma t}t^{-\varepsilon}\|_{L^1(0,T)}=\int_0^Te^{-\sigma t}t^{-\varepsilon}dt\leq \sigma^{\varepsilon-1}\int_0^\infty e^{- t}t^{-\varepsilon}dt\leq Q\sigma^{\varepsilon-1}.
$
 We use the Laplace transform to evaluate the second right-hand side term of (\ref{mh9}) to get
\begin{equation}\label{Well:e2}
\begin{array}{l}
\ds \hspace{-0.1in} \mathcal L \bigg [ \int_0^{t}  F'(t-s) \p_sp_v(s)  d s\bigg ]
= \big(z^{\alpha_0-\varepsilon}(z^{\alpha_0}-\kappa A)^{-1}\big)\big(z^\varepsilon \mathcal L(\p_sp_v)\big).
\end{array}
\end{equation}
Take the inverse Laplace transform of \eqref{Well:e2}  to get
$
 \int_0^{t}  F'(t-s) \p_s p_v(s)  d s  = \int_0^t \mathcal R(t-s) \partial_s^\varepsilon \p_sp_v(s)d s
$
where
$\mathcal R(t)$ bounded by \eqref{Well:e4}.

We invoke this and apply (\ref{v3}) and (\ref{mh4}) to bound the second right-hand side term of (\ref{mh9}) as
\begin{align*}
&\Big\|\int_0^tF'(t-s)\p_sp_v(s)ds\Big\|\\
&\qquad\leq Qt^{\ve-1}*[\beta_{\alpha_0-\ve-1}\|f(0)\|+\beta_{\alpha_0-\ve-1}*\|\p_tf\|+t^{-2\ve}*\|\p_t^2v\|]\\
&\qquad\leq Q\beta_{\alpha_0-1}\|f(0)\|+Q\beta_{\alpha_0-1}*\|\p_tf\|+Q\beta_{1-\ve}*\|\p_t^2v\|,\label{v10}
\end{align*}
which leads to
\begin{align}
&\Big\|e^{-\sigma t}\Big\|\int_0^tF'(t-s)\p_sp_v(s)ds\Big\|\Big\|_{L^p(0,T)}\leq Q\|f\|_{W^{1,1}(L^2)}+Q\sigma^{\ve-1}\|v\|_{\tilde W^{2,p}(L^2)}.
\end{align}

 We summarize the above estimates \eqref{vmh9}--\eqref{v10} in (\ref{mh9}) to conclude that
 \begin{equation}\label{mh10}
 \|w\|_{\tilde W^{2,p}(L^2)}\leq Q\big(\|f\|_{W^{1,1}(L^2)}+\sigma^{\varepsilon-1}\|v\|_{\tilde W^{2,p}(L^2)}\big),
 \end{equation}
 which implies that $w\in \tilde W^{2,p}(L^2)$ such that $\mathcal M$ is well-posed. To show the contractivity of $\mathcal M$, let $w_1=\mathcal M v_1$ and $w_2=\mathcal M v_2$ such that $w_1-w_2$ satisfies (\ref{Modelw}) with $v=v_1-v_2$ and $f\equiv 0$. Thus (\ref{mh10}) implies $\|w_1-w_2\|_{\tilde W^{2,p}(L^2)}\leq Q\sigma^{\varepsilon-1}\|v_1-v_2\|_{\tilde W^{2,p}(L^2)}$. Choose $\sigma$ large enough such that $Q\sigma^{\varepsilon-1}<1$, that is, $\mathcal M$ is a contraction mapping such that there exists a unique solution $u\in \tilde W^{2,p}(L^2)$ for model (\ref{Model2}) and (\ref{ibc}) with $u_0=\bar u_0\equiv 0$, and the stability estimate could be derived directly from (\ref{mh10}) with $v=w=u$ and large $\sigma$ and the equivalence between two norms $\|\cdot\|_{\tilde W^{2,p}(L^2)}$ and $\|\cdot\|_{W^{2,p}(L^2)}$ for $u\in \tilde W^{2,p}(L^2)$
  \begin{equation}\label{mh11}
 \|u\|_{ W^{2,p}(L^2)}\leq Q\|f\|_{W^{1,1}(L^2)}.
 \end{equation}


For (\ref{Model2}) and (\ref{ibc}) with non-zero initial conditions, a variable substitution $v=u-u_0-t\bar u_0$ could be used to reach the same model with $u$, $f$, $u_0$, $\bar u_0$ replaced by $v$, $f+\kappa A u_0+t\kappa A\bar u_0$, $0$ and $0$, respectively. As $u_0,\bar u_0\in H^2$, we apply the well-posedness of (\ref{Model2}) and (\ref{ibc}) with homogeneous initial conditions to find that there exists a unique solution $v\in \tilde W^{2,p}(L^2)$ with the stability estimate
$
 \|v\|_{ W^{2,p}(L^2)}\leq Q\|f+\kappa A u_0+t\kappa A \bar u_0\|_{W^{1,1}(L^2)}
$
 derived from (\ref{mh11}).
 By $v=u-u_0-t\bar u_0$, we finally conclude that $u=v+u_0+t\bar u_0$ is a solution to (\ref{Model2}) and (\ref{ibc}) in $W^{2,p}(L^2)$ with the estimate
\begin{equation}\label{ima1}
 \|u\|_{ W^{2,p}(L^2)}\leq Q\|f\|_{W^{1,1}(L^2)}+Q\|u_0\|_{H^2}+Q\|\bar u_0\|_{H^2}.
\end{equation}
 The uniqueness of the solutions to (\ref{Model2}) and (\ref{ibc}) follows from that to this model with $u_0=\bar u_0\equiv 0$.

Then we show that the solution to (\ref{Model2}) and (\ref{ibc}) is also a solution to (\ref{VtFDEs})--(\ref{ibc}). Note that (\ref{Model2}) implies (\ref{modeln}) by integration by parts, which, together with $g=\beta_{\alpha_0-1}*k$, leads to
$
\beta_{\alpha_0-1}*\big[k*\p_t^2u -\kappa  A u-f\big]=0.
$
 The convolution of this with $\beta_{2-\alpha_0}$ is
 \vspace{-0.04in}
 \begin{align}\label{v5}
\beta_{1}*\big[k*\p_t^2u - \kappa A u-f\big]=0.
\end{align}
 Differentiate this equation leads to the original governing equation (\ref{VtFDEs}), which implies the existence. The uniqueness of the solutions to (\ref{VtFDEs})--(\ref{ibc}) in $W^{2,p}(L^2)$ follows from that of (\ref{Model2}) and (\ref{ibc}).  Finally, we apply the governing equation (\ref{VtFDEs}) and (\ref{ima1}) to bound
$\| A u\|_{L^p(L^2)}$,
which completes the proof.


\vskip 4mm
{\footnotesize \noindent \textbf{Acknowledgments} The authors are grateful to the reviewers for their helpful suggestions to enhance the quality of the paper.
}

\vskip 4mm
{\footnotesize \noindent \textbf{Funding} This work was partially supported by the National Natural Science Foundation of China (12301555), by the National Key R\&D Program of China (2023YFA1008903), by the Taishan Scholars Program of Shandong Province (tsqn202306083), and by the Postdoctoral Fellowship Program of CPSF (GZC20240938).
}

\vskip 4mm
{\footnotesize\noindent \textbf{Data Availability} The datasets are available from the corresponding author upon reasonable request.
}

\section*{Declarations}

{\footnotesize\noindent \textbf{Conflict of interest} The authors declare that they have no known competing financial interests or personal relationships that could have appeared to influence the work reported in this paper.
}



\end{document}